\documentclass[10]{article}
\usepackage{psfrag}
\usepackage{epstopdf}
\usepackage{epsfig}
\usepackage{verbatim}
\usepackage{fancyhdr}
\usepackage{subfigure}
\usepackage{indentfirst}
\usepackage{booktabs}
\usepackage{color}
\usepackage{mathrsfs}
\usepackage{graphics}
\usepackage{graphicx}
\usepackage{listings}
\usepackage{xcolor}
\usepackage{float}
\usepackage{multirow}

\usepackage{amsmath}
\usepackage{amssymb}
\usepackage{amscd}
\usepackage{amsthm}

\usepackage{algorithm}  
\usepackage{algorithmic}  
\usepackage{enumerate}

\usepackage{xypic}  

\usepackage{My_macro_unified}
\usepackage{My_macro_special}

\usepackage{hyperref}
\hypersetup{hypertex=true,
	colorlinks=true,
	linkcolor=blue,
	anchorcolor=blue,
	citecolor=blue}

\begin{document}
	
%\title{\textbf{Solving discrete constrained problems on de Rham complex}}

\title{\textbf{Solving discrete constrained problems on de Rham complex}}

\author{ Zhongjie Lu 
	\thanks{School of Mathematical Sciences, 
		University of Science and Technology of China, Hefei, Anhui 230026, China.
	Email: zhjlu@ustc.edu.cn. Research supported by NSFC grant No. 12101586.}}

\date{}
%\date{\color{red}\today}
	
	\maketitle
	
	\begin{abstract}
		
		The poor and even ill conditions of the discrete constrained problems
		cause some difficulties in solving them with iterative methods.
		In this paper, 
		we transform the discrete constrained problems on de Rham complex
		to equivalent Laplace-like problems.
		This transformation not only make these constrained problems solvable,
		but also make it easy to use many existing iterative methods
		and preconditioning techniques to solving large-scale 
		discrete constrained problems.
		
	\textbf{Keywords: constrained problem, divergence-free,
		ill condition, iterative method }

	\end{abstract}

%\clearpage

\section{Introduction}

Many systems of partial differential equations
contain constraint conditions.
For example,
when solving the second-order Maxwell equation,
we want the solution to satisfy 
the divergence-free condition $ \nabla \cdot \vu = 0 $.
To this end, 
a Lagrange multiplier is inserted into the equation
\begin{equation}\label{mixed_Maxwell}
	\begin{split}
		\nabla \times \nabla \times \vu + c \vu + \nabla p &= \vf, \\
		\nabla \cdot \vu &= 0.
	\end{split}
\end{equation}
Then the component in the solution
that dose not satisfy the condition $ \nabla \cdot \vu = 0 $
is eliminated.
The weak formulation for such problems usually
have the following form:
\begin{equation}\label{weak_form}
	\begin{split}
		a(\vu,\vv) + b(\vv,p) &= (\vf,\vv) \,\qq \forall \vv \in V, \\
		b(\vu,q)   &= 0 \qqq\qq \forall q \in Q.
	\end{split}
\end{equation}
Here, $ V $ and $ Q $ are the function spaces for $ \vu $ and $ p $, respectively.
By constructing two proper finite element spaces $ V_h $ and $ Q_h $ for the two spaces,
we obtain its discrete problem
\begin{equation}\label{discrete_weak}
	\begin{split}
		a(\vu_h,\vv) + b(\vv_h,p_h) &= (\vf,\vv_h) \qq \forall \vv_h \in V_h, \\
		b(\vu_h,q_h)   &= 0 \qqq\qq\;\, \forall q_h \in Q_h.
	\end{split}
\end{equation}
Usually this kind of discrete problems has poor condition.
This results in that 
many existing iterative methods and preconditioning techniques 
that are efficient for Laplace-like problems
are inefficient and even invalid for such problems.
With some boundary conditions 
and domains with some special shapes,
we can prove that
the terms $ b(\vu,q) $ in \eqref{weak_form} 
and $ b(q_h,\vu_h) $ in \eqref{discrete_weak} 
no longer satisfy the inf-sup condition,
but $ \vu $ in \eqref{weak_form} and $ \vu_h $ in \eqref{discrete_weak}
 are still unique.
But the algebraic system is non-invertible.
Even direct methods can not solve it.
Thing is worse 
for the constrained grad-div problem
\begin{equation}\nonumber %\label{mixed_grad_div}
	\begin{split}
		-\nabla ( \nabla \cdot \vu) + c \vu + \nabla \times  \vp &= \vf, \\
		\nabla \times \vu &= 0.
	\end{split}
\end{equation}
There is an infinite-dimensional kernel contained in 
its constraint condition $ \nabla \times \vu = 0 $.
Its corresponding algebraic system is highly ill.

A popular method to deal with such discrete problems is 
penalty method \cite{boffi2013mixed}.
If the $ \vu_h $ in \eqref{discrete_weak} is unique,
the system of the penalty method is solvable.
With the penalty parameter becoming large,
the penalty solution tends to the exact solution.
If we want more precious solution,
the penalty parameter should be larger.
However, if the penalty parameter is too large,
the condition of the system of penalty method become terrible.
Besides this, because of the limited machine precision,
the penalty parameter can be arbitrarily large.
There are also researches on the preconditioned discretizations 
for partial differential equations that can match such problems \cite{MR2769031}.
There is a comprehensive survey about 
the saddle point problems in \cite{MR2168342}.

The Maxwell operator and grad-div operator are the 
$ k=1 $ and $ k=2 $ forms of $ d^*d $ operator on 
$ \mbR^3 $ complex
\begin{equation}\nonumber %\label{R3_complex_0}
	\begin{split}
		\xymatrix{
			0 \ar[r] & H(\textgrad) \ar[r]^{\nabla} 
			& \vH(\textcurl) \ar[r]^{\nabla \times} 
			& \vH(\textdive) \ar[r]^{\nabla \cdot} 
			& L^2 \ar[r]& 0.
		}
	\end{split}
\end{equation}
For the general de Rham complex
\begin{equation}\nonumber %\label{deRham_complex_0}
	\begin{split}
		\xymatrix{
			0 \ar[r] & V^0 \ar[r]^{d^0} 
			& V^1 \ar[r]^{d^1} 
			& \cdots \ar[r]^{d^{n-1}} 
			& V^n \ar[r]& 0,
		}
	\end{split}
\end{equation}
these constrained problems can be written in a uniform formulation
\begin{equation}\label{mixed_dd_0}
	\begin{split}
		(d^k)^* d^k \vu + c \vu + d^{k-1} \vp &= \vf, \\
		(d^{k-1})^* \vu &= \vg.
	\end{split}
\end{equation}

In this paper, we study the discrete problem of 
the constrained problem \eqref{mixed_dd_0}
in the uniform framework of de Rham complex.
Using the property of complex,
we construct several equivalent problems for 
the corresponding discrete problem \eqref{discrete_weak} 
for \eqref{mixed_dd_0}.
No mater whether the discrete system \eqref{discrete_weak}  
is invertible or not,
only if the component $ \vu_h $ is unique,
we can solve it though some well-posed problems
that we construct.
Furthermore,
the spectral distributions of the equivalent problems are Laplace-like.
It is easy to use
many existing iterative methods
and preconditioning techniques 
to solve large-scalar discrete constrained problems.

This paper is organized as follows.
In Section \ref{sec_2},
we set up the matrix constrained system and 
derive some theoretical results 
that we shall use in the later sections.
In Section \ref{sec_3} and Section \ref{sec_4},
we construct equivalent problems for the case 
$ \vg = 0 $ and $ \vg \not = 0 $ in \eqref{mixed_dd_0}, respectively.
In Section \ref{sec_5},
we study the discretization of the constrained problems on de Rham complex
and illustrate its relationship with the matrix systems 
discussed in previous sections.
In Section \ref{sec_6},
we consider some problems that are involved in solving the equivalent problems.
In Section \ref{sec_7},
we take the constrained Maxwell and grad-div problems as numerical examples
to verify the equivalent problems.
There are some conclusions in Section \ref{sec_8}.

\section{Matrix aspects of the constrained problem} \label{sec_2}

Before we consider the discretization of the constrained problems on de Rham complex,
we study their matrix form:
\begin{equation}\label{mixed_matrix}
	\begin{split}
		(\mcA + c\mcM)u + \mcB p &= \mcF, \\
		\mcB^T u &= \mcG.
	\end{split}
\end{equation}
In this section, we definite this matrix system
and derive some theoretical results 
that we shall use in the later sections.
In Section \ref{sec_5},
we will show the relation between 
the finite element discretization 
and this system.

In the system \eqref{mixed_matrix}, 
it involves the following matrices and vectors:
\begin{definition}
	Let $ N,M $ be two positive integers.
	Let
	$\mcA \in \mbC^{N\times N}$, 
	$ \mcB \in \mbC^{N\times M}$,
	$ \mcM \in \mbC^{N\times N}$,
	$ \mcF \in \mbC^{N \times 1} $,
	$ \mcG \in \mbC^{M \times 1} $,
	$ u \in \mbC^{N \times 1} $,
	$ p \in \mbC^{M \times 1} $
	and  $ c \geq 0 $ be a nonnative real number.
	Here $ \mcA $ is Hermitian positive semi-definite
	and $ \mcM $ is Hermitian positive definite.
\end{definition}
\noindent There is an assumption on $ \mcA $, $ \mcM $ and $ \mcB $.
\begin{assumption}\label{AMB0}
	$ \mcM $-complex property:
	\begin{equation}\nonumber%\label{}
		\begin{split}
			\mcA\mcM^{-1}\mcB = 0.
		\end{split}
	\end{equation}
\end{assumption}
\noindent 
This is the main difference between the system \eqref{mixed_matrix}
and general systems with constraint conditions.
This property corresponds to the basic property $ d^{k}d^{k-1}=0 $ on complexes.
This is the reason that we call this property 'complex'.
When solving the system \eqref{mixed_matrix},
it also involves an extra matrix:
\begin{definition}
	$ \mcU \in \mbC^{M\times M} $ is a Hermitian positive definite matrix.
\end{definition}

We consider the Hermitian generalized eigenvalue problem
\begin{equation}\label{A_eig}
	\begin{split}
		\mcA u  = \lambda \mcM u,\\
	\end{split}
\end{equation}
According to its eigenvectors, 
the entire space $ \mbC^N $ can be divided into 
two $ \mcM $-orthogonal parts:
\begin{equation}\nonumber%\label{}
	\begin{split}
		\mbC^{N} = \mbC_1 \oplus_{\mcM} \Ker \mcA.
	\end{split}
\end{equation}
Here $ \mbC_1 $ is spanned by the eigenvectors with nonzero eigenvalues of \eqref{A_eig}.
The set of nonzero eigenpairs of \eqref{A_eig} is denoted by
\begin{equation}\label{A_eigenpairs}
	\begin{split}
		\left\lbrace \left(\lambda^{(1)}_i, \mfu^{(1)}_i \right) \right\rbrace 
		_{i=1}^{\dim \mbC_1}.
	\end{split}
\end{equation}
The set 
\begin{equation}\label{base_C1}
	\begin{split}
		\left\lbrace  \mfu^{(1)}_i  \right\rbrace 
		_{i=1}^{\dim \mbC_1}.
	\end{split}
\end{equation}
can be a base of the subspace $ \mbC_1 $.
Similar to \eqref{A_eig},
for the Hermitian generalized eigenvalue problem
\begin{equation}\label{B_eig}
	\begin{split}
		\mcB\mcU\mcB^T u = \lambda \mcM u,
	\end{split}
\end{equation}
there is another  $ \mcM $-orthogonal decomposition for $ \mbC^N $:
\begin{equation}\nonumber%\label{}
	\begin{split}
		\mbC^{N} = \mbC_2 \oplus_{\mcM} \Ker\mcB\mcU\mcB^T,
	\end{split}
\end{equation}
where $ \mbC_2 $ is spanned by the eigenvectors of the nonzero eigenvalues of \eqref{B_eig}.
The set of nonzero eigenpairs of \eqref{B_eig} is denoted by
\begin{equation}\label{B_pairs}
	\begin{split}
		\left\lbrace \left(\lambda^{(2)}_i, \mfu^{(2)}_i \right) \right\rbrace 
		_{i=1}^{\dim \mbC_2}.
	\end{split}
\end{equation}
The set 
\begin{equation}\label{B_base}
	\begin{split}
		\left\lbrace  \mfu^{(2)}_i  \right\rbrace 
		_{i=1}^{\dim \mbC_2}.
	\end{split}
\end{equation}
can be a base of the subspace $ \mbC_2 $.

For an eigenpair 
$   \left(\lambda^{(2)}_i, \mfu^{(2)}_i \right) $
in the set \eqref{B_pairs},
we have 
\begin{equation}\nonumber%\label{}
	\begin{split}
		\frac{1}{\lambda^{(2)}_i}\mcM^{-1}\mcB \mcU \mcB^T \mfu^{(2)}_i = \mfu^{(2)}_i.
	\end{split}
\end{equation}
By Assumption \ref{AMB0}, we have
\begin{equation}\nonumber%\label{}
	\begin{split}
		\mcA \mfu^{(2)}_i  = \frac{1}{\lambda^{(2)}_i}\mcM^{-1}\mcA\mcB \mcU \mcB^T \mfu^{(2)}_i
		=0,
	\end{split}
\end{equation}
and we obtain
\begin{equation}\nonumber%\label{}
	\begin{split}
		\mfu^{(2)}_i \in \Ker \mcA.
	\end{split}
\end{equation}
As the set \eqref{B_base} is a base of $ \mbC_2 $,
we have
\begin{equation}\nonumber%\label{}
	\begin{split}
		\mbC_2\subset \Ker \mcA.
	\end{split}
\end{equation}

We denote the intersection of $ \Ker \mcA $ and $ \Ker\mcB\mcU\mcB^T $ by
\begin{equation}\nonumber%\label{}
	\begin{split}
		\mbC_0 \triangleq \Ker \mcA \cap \Ker\mcB\mcU\mcB^T.
	\end{split}
\end{equation}
The $ \mbC_0 $ is the eigenspace of the zero eigenvalue 
of the generalized eigenvalue problem
\begin{equation}\label{AB_eig}
	\begin{split}
		\left( \mcA + \mcB\mcU\mcB^T \right) u = \lambda \mcM u.
	\end{split}
\end{equation}
Then the entire space $ \mbC^N $  can be decomposed into three $ \mcM $-orthogonal parts:
\begin{equation}\label{decomp_M}
	\begin{split}
		\mbC^{N} = 
		\mbC_0 \oplus_{\mcM} \mbC_1 \oplus_{\mcM}\mbC_2.
	\end{split}
\end{equation}
Consequently, 
there is a  $ \mcM^{-1} $-orthogonal decomposition for $ \mbC^N $:
\begin{equation}\nonumber%\label{}
	\begin{split}
		\mbC^{N} = \mcM \mbC_0 \oplus_{\mcM^{-1}}\mcM\mbC_1 \oplus_{\mcM^{-1}}\mcM\mbC_2.
	\end{split}
\end{equation}
The two kernels can be presented by 
\begin{equation}\label{kernelAB}
	\begin{split}
		\Ker \mcA = \mbC_0  \oplus_{\mcM}\mbC_2
		\qq \text{and} \qq
		\Ker \mcB\mcU\mcB^T = \mbC_0 \oplus_{\mcM}\mbC_1.
	\end{split}
\end{equation}

\begin{definition}\label{H_def}
	The set $ \left\lbrace \mfu^{(0)}_i \right\rbrace_{i=1}^{\dim \mbC_0} $
	is a $ \mcM $-orthogonal base of the subspace $ \mbC_0 $.
	There is $ \mfu^{(0)T}_i \mcM \mfu^{(0)}_i = 1 $ for $ i =1,\cdots, \dim\mbC_0 $.
	The matrix $ \mcH \in \mbC^{N \times \dim\mbC_0} $ is constructed by this base
	$ \mcH = \left[\mfu^{(0)}_1, \mfu^{(0)}_2, \cdots, \mfu^{(0)}_{\dim \mbC_0} \right]  $.
\end{definition}
\noindent 
The set $ \left\lbrace \mfu^{(0)}_i \right\rbrace_{i=1}^{\dim \mbC_0} $
can be gotten through 
computing the eigenvectors of zero eigenvalue in
the eigenvalue problem \eqref{AB_eig}.
We consider the generalized eigenvalue problem
\begin{equation}\label{H_eig}
	\begin{split}
		\mcM \mcH \mcH^T \mcM u = \lambda \mcM u.
	\end{split}
\end{equation}
By Definition \ref{H_def},
for a $ \mfu^{(0)}_i $ in the set
$ \left\lbrace \mfu^{(0)}_i \right\rbrace_{i=1}^{\dim \mbC_0} $,
we have
\begin{equation}\nonumber%\label{}
	\begin{split}
		\mcM \mcH \mcH^T \mcM \mfu^{(0)}_i = 
		\left( u^{(0)T}_i \mcM \mfu^{(0)}_i \right)  \mcM \mfu^{(0)}_i
		= \mcM u^{(0)}_i.
	\end{split}
\end{equation}
Then the set of nonzero eigenpairs of \eqref{H_eig} is
\begin{equation}\label{H_pairs}
	\begin{split}
		\left\lbrace\left(1, \mfu^{(0)}_i \right) \right\rbrace 
		_{i=1}^{\dim \mbC_0}.
	\end{split}
\end{equation}
By the $ \mcM $-orthogonal decomposition \eqref{decomp_M},
we know
\begin{equation}\nonumber%\label{}
	\begin{split}
		\Ker \mcM \mcH \mcH^T \mcM = \mbC_1 \oplus_{\mcM}\mbC_2.
	\end{split}
\end{equation}
Summarizing the matrix operators and the subspaces, 
we have the following theorem.
\begin{theorem}\label{invariant}
	\begin{equation}\nonumber %label{}
		\begin{split}
			\mcA \mbC_0 &= \{0\}, 
			\qqq \qqq\qqq \mcA \mbC_2 = \mcM \mbC_1,
			\qqq\qqq \mcA \mbC_2 = \{0\},\\
			\mcB\mcU\mcB^T \mbC_0 &= \{0\}, 
			\qqq\qqq \mcB\mcU\mcB^T \mbC_2 = \{0\},
			\qqq\qq \mcB\mcU\mcB^T \mbC_2 = \mcM \mbC_2,\\
			\mcM \mcH \mcH^T \mcM \mbC_0 &= \mcM \mbC_0, 
			\qq\, \mcM \mcH \mcH^T \mcM \mbC_2 = \{0\},
			\qq \mcM \mcH \mcH^T \mcM \mbC_2 = \{0\}.
		\end{split}
	\end{equation}
	
\end{theorem}

\begin{theorem}\label{unique_solution}
	Let $ c \geq 0 $.
	For $ f^{(1)} \in \mbC_1 $,
	there a unique solution $ u^{(1)} $ in the subspace $ \mbC_1 $
	satisfying $ \mcA u^{(1)} + c \mcM u^{(1)} = \mcM f^{(1)} $.
	For $ f^{(2)} \in \mbC_2 $,
	there a unique solution $ u^{(2)} $ in the subspace $ \mbC_2 $
	satisfying $ \mcB^T\mcU\mcB u^{(2)} + c \mcM u^{(2)} = \mcM f^{(2)} $.
	For $ f^{(0)} \in \mbC_0 $,
	there a unique solution $ u^{(0)} $ in the subspace $ \mbC_0 $
	satisfying $ \mcM \mcH \mcH^T \mcM u^{(0)}  = \mcM f^{(0)} $.
\end{theorem}

\begin{proof}
	As $ \mcA $ is positive semi-definite, 
	the eigenvalues in the nonzero
	eigenpair set \eqref{A_eigenpairs}  are all positive.
	As $ f^{(1)} \in \mbC_1 $,
	using the base \eqref{base_C1},
	it can be expanded as
	\begin{equation}\label{f1_decomp1}
		\begin{split}
			f^{(1)} = \sum_{i=1}^{\dim \mbC_1} f^{(1)}_i \mfu^{(1)}_i.
		\end{split}
	\end{equation}
	Here $ f^{(1)}_i \in \mbC $ is the coefficient of each eigenvector 
	in the base \eqref{base_C1}.
	Similarly, $ u^{(1)} \in \mbC_1 $ can be expanded as
	\begin{equation}\label{u1_decomp1}
		\begin{split}
			u^{(1)} = \sum_{i=1}^{\dim \mbC_1} u^{(1)}_i \mfu^{(1)}_i,
		\end{split}
	\end{equation}
	where $ u^{(1)}_i \in \mbC $ is the coefficient to be computed.
	
	Putting the expansions
	\eqref{f1_decomp1} and \eqref{u1_decomp1} into the equation,
	we have
	\begin{equation}\nonumber%\label{}
		\begin{split}
			\left( \mcA + c\mcM\right) \sum_{i=1}^{\dim \mbC_1} u^{(1)}_i \mfu^{(1)}_i = \mcM \sum_{i=1}^{\dim \mbC_1} f^{(1)}_i \mfu^{(1)}_i.
		\end{split}
	\end{equation}
	By the equation $ \mcA \mfu^{(1)}_i =\lambda^{(1)}_i\mcM \mfu^{(1)}_i $, 
	we have
	\begin{equation}\nonumber%\label{}
		\begin{split}
			\sum_{i=1}^{\dim \mbC_1} 
			\left( \lambda^{(1)}_i + c\right)u^{(1)}_i \mcM\mfu^{(1)}_i =  \sum_{i=1}^{\dim \mbC_1}  f^{(1)}_i \mcM\mfu^{(1)}_i.
		\end{split}
	\end{equation}
	Comparing the coefficients on both sides,
	we obtain
	\begin{equation}\nonumber%\label{}
		\begin{split}
			\left(\lambda^{(1)}_i + c \right) u^{(1)}_i = f^{(1)}_i \qq \text{for each } i.
		\end{split}
	\end{equation}
	Then we have
	\begin{equation}\nonumber%\label{}
		\begin{split}
			u^{(1)}_i = \frac{f^{(1)}_i}{ \lambda^{(1)}_i + c }
		\end{split}
	\end{equation}
	and we obtain the unique solution $ u^{(1)} \in \mbC_1 $ 
	for the equation $ \mcA u^{(1)} + c \mcM u^{(1)} = \mcM f^{(1)} $. 
	
	The proofs for the next two results are similar.
\end{proof}

\begin{theorem}\label{kerBandkerBB}
	$ \Ker\mcB^T = \Ker\mcB\mcU\mcB^T = \mbC_0 \oplus_{\mcM}\mbC_1$.
	% 	\begin{equation}\nonumber%\label{}
	% 		\begin{split}
	% 			%
	% 			\Ker\mcB^T = \Ker\mcB\mcU\mcB^T = \mbC_0 \oplus_{\mcM}\mbC_1.
	% 			%
	% 		\end{split}
	% 	\end{equation}
\end{theorem}
\begin{proof}
	For $ u \in \Ker\mcB^T $,
	we have  $ \mcB\mcU\mcB^T u = \mcB\mcU(\mcB^T u) = 0 $
	and then $ u \in \Ker\mcB\mcU\mcB^T $.
	
	For $ u \in \Ker\mcB\mcU\mcB^T $,  we have
	\begin{equation}\nonumber%\label{}
		\begin{split}
			\mcB\mcU\mcB^T u = 0 
			\Longrightarrow u^T \mcB\mcU\mcB^T u = 0
			\Longrightarrow (\mcB^T u )^T\mcU(\mcB^T u) = 0.
		\end{split}
	\end{equation}
	As $ \mcU $ is Hermitian positive definite,
	we have $ \mcB^T u = 0 $ and then $ u \in \Ker\mcB^T $.
\end{proof}

\begin{theorem}\label{CM_decomp}
	$ \mbC^M = \Ker \mcB \oplus \IM \mcB^T $.
\end{theorem} 

\begin{proof}
	
	There is an orthogonal decomposition for $ \mbC^M $:
	\begin{equation}\label{CM_decomp_1}
		\begin{split}
			\mbC^M = \IM \mcB^T \oplus \left( \IM \mcB^T \right) ^\perp.
		\end{split}
	\end{equation}
	
	For $ p \in \left( \IM \mcB^T \right) ^\perp $,
	we have 
	\begin{equation}\nonumber %label{}
		\begin{split}
			(\mcB p , v) = (p , \mcB^T v) = 0 \qq \text{for any  } v\in \mbC^N.
		\end{split}
	\end{equation}
	Here we use $ \mcB^T v \in \IM \mcB^T $ for any $ v \in \mbC^N $.
	Then we have $ \mcB p = 0 $ and $ \left( \IM \mcB^T \right) ^\perp \subset \Ker \mcB $.

	For $ p\in \IM \mcB^T $,
	there is a $ v \in \mbC^N $ such that 
	$ p = \mcB^T v $.
	For $ q\in \Ker \mcB $,
	we have
	\begin{equation}\nonumber %label{}
		\begin{split}
			(p,q)= (\mcB^T v,p) = ( v, \mcB q) = 0.
		\end{split}
	\end{equation}
	Then we have $ \Ker \mcB \perp \IM \mcB^T $.
	By the decomposition \eqref{CM_decomp_1},
	we have $ \Ker \mcB \subset \left( \IM \mcB^T \right) ^\perp $.
	
	Then we have  $ \Ker \mcB = \left( \IM \mcB^T \right) ^\perp $.
	By the decomposition \eqref{CM_decomp_1},
	we obtain the conclusion.
\end{proof}

With the similar proof, we have the next theorem.
\begin{theorem}\label{CN_decomp}
	$ \mbC^N = \Ker \mcB^T \oplus \IM \mcB $.
\end{theorem}

\begin{theorem}\label{ImB}
	$ \mbC_2 = \IM \mcM^{-1}\mcB $ and	$\mcM\mbC_2 = \IM \mcB $.
\end{theorem}

\begin{proof}
	By Theorem \ref{CN_decomp}, we have
	\begin{equation}\label{CN_decomp_2}
		\begin{split}
			\mbC^N = \Ker \mcB^T \oplus \IM \mcB
			= \Ker \mcB^T \oplus_{\mcM} \mcM^{-1}\IM \mcB
			= \Ker \mcB^T \oplus_{\mcM} \IM \mcM^{-1}\mcB.
		\end{split}
	\end{equation}
	By Theorem \ref{kerBandkerBB}, we have
	\begin{equation}\nonumber%\label{}
		\begin{split}
			\Ker\mcB^T = \Ker\mcB\mcU\mcB^T = \mbC_0 \oplus_{\mcM}\mbC_1.
		\end{split}
	\end{equation}
	By the $ \mcM $-decomposition for $ \mbC^N $ \eqref{decomp_M} and \eqref{CN_decomp_2},
	we obtain $ \mbC_2 = \IM \mcM^{-1}\mcB $.

	The second result $\mcM\mbC_2 = \IM \mcB $ is equivalent to the first one.
\end{proof}

\begin{theorem}\label{BTGU}
	For $ \mcG \in \IM \mcB^T $,	
	if there is a $ u_g\in\mbC^N $ such that $ \mcB \mcU \mcB^T u_g = \mcB \mcU \mcG $,
	then $ \mcB^T u_g = \mcG $.
\end{theorem}

\begin{proof}
	Letting $ \tilde \mcG = \mcB^T u_g  $,
	then we have $ \tilde \mcG \in \IM \mcB^T $
	and $ \mcB \mcU \mcB^T u_g = \mcB \mcU \tilde\mcG $.
	Consequently, we have $ \mcG - \tilde\mcG \in \IM \mcB^T $
	and $ \mcB \mcU (\mcG - \tilde\mcG) = 0 $.
	Then $ \mcU (\mcG - \tilde\mcG) \in \Ker \mcB $.
	As $ \Ker \mcB \perp \IM \mcB^T $ by Theorem \ref{CM_decomp},
	we have $ (\mcG - \tilde\mcG)^T \mcU (\mcG - \tilde\mcG) = 0 $.
	As $ \mcU $ is a Hermitian positive definite matrix,
	we have $ \mcG - \tilde\mcG = 0 $.
\end{proof}

%\clearpage

\section{ The case $ \mcG = 0 $} \label{sec_3}

In this section, we consider the solution of \eqref{mixed_matrix} in the case $ \mcG = 0 $:
\begin{equation}\label{mixed_matrix_g0}
	\begin{split}
		(\mcA + c\mcM)u + \mcB p &= \mcF,\\
		\mcB^T u &= 0.
	\end{split}
\end{equation}

\subsection{The case $ \dim \mbC_0 = 0 $}

In this case, the subspace $ \mbC_0 $ vanishes.
The $ \mcM $-orthogonal decomposition \eqref{decomp_M} for $ \mbC^N $ becomes
\begin{equation}\label{decomp_M_C0}
	\begin{split}
		\mbC^{N} = \mbC_1 \oplus_{\mcM}\mbC_2.
	\end{split}
\end{equation}
The vector $ \mcM^{-1} \mcF $ can be divided into two $ \mcM $- orthogonal parts:
\begin{equation}\label{F_decomp}
	\begin{split}
		\mcM^{-1} \mcF  \triangleq  f^{(1)} + f^{(2)}.
	\end{split}
\end{equation}
Here $ f^{(1)} \in \mbC_1 $ and $ f^{(2)}\in\mbC_2 $.
Consequently, 
the term $ \mcF $ can be divided into two $ \mcM^{-1} $- orthogonal parts:
\begin{equation}\label{F_decomp_C0}
	\begin{split}
		\mcF = \mcM f^{(1)} + \mcM f^{(2)}, 	
	\end{split}
\end{equation}
where $ \mcM f^{(1)} \in \mcM\mbC_1 $ and $ \mcM f^{(2)}\in \mcM\mbC_2 $.

By the constraint condition $ \mcB^T u = 0 $ in the system \eqref{mixed_matrix_g0},
we know that $ u \in \Ker \mcB^T $.
As $ \dim \mbC_0 = 0 $, 
by Theorem \ref{kerBandkerBB}, we have
\begin{equation}\nonumber%\label{}
	\begin{split}
		\Ker \mcB^T = \mbC_1.
	\end{split}
\end{equation}
Then the solution $ u\in \mbC^N $ in the system \eqref{mixed_matrix_g0} 
satisfies $ u\in \mbC_1 $ and can be written as
\begin{equation}\label{u_decomp}
	\begin{split}
		u \triangleq u^{(1)},
	\end{split}
\end{equation} 
where $ u^{(1)} \in \mbC_1 $.
Putting the decompositions \eqref{F_decomp} and \eqref{u_decomp} 
into the first equation of \eqref{mixed_matrix_g0},
we have
\begin{equation}\label{u_system}
	\begin{split}
		(\mcA + c\mcM)u + \mcB p &= \mcF, \\
		( \mcA + c\mcM) u^{(1)} +  \mcB p &=  \mcM f^{(1)} + \mcM f^{(2)}, \\
		\underbrace{ (\mcA + c\mcM) u^{(1)} - \mcM f^{(1)} }_{\mcM\mbC_1}  
		&=\underbrace{ \mcM f^{(2)} - \mcB p }_{\mcM\mbC_2}.
	\end{split}
\end{equation}
By Theorem \ref{invariant} 
and $ \mcB p \in \IM \mcB = \mcM \mbC_2 $ in Theorem \ref{ImB},
we know that 
the two parts in the equation above belong to 
two $ \mcM^{-1} $-orthogonal subspaces, respectively.
Then, to make this equation hold,
both parts should be zero.
By Theorem \ref{unique_solution},
there is a unique solution $ u^{(1)} \in \mbC_1 $ satisfying
\begin{equation}\label{u_1_component}
	\begin{split}
		(\mcA + c\mcM) u^{(1)} = \mcM f^{(1)}.
	\end{split}
\end{equation}
Then by decomposition \eqref{u_decomp},
the solution of $ u\in \mbC^N $ in the system \eqref{mixed_matrix_g0} 
is unique in the case $ \dim \mbC_0 = 0 $.
By Theorem \ref{ImB},
there exists a $ p\in \mbC^M $ such that
\begin{equation}\label{p_Mf2}
	\begin{split}
		\mcB p = \mcM f^{(2)}.
	\end{split}
\end{equation}
If the kernel $ \Ker \mcB $ in Theorem \ref{CM_decomp} is trivial
or in other words, the columns of $ \mcB $ are full-rank,
$ p $ in \eqref{mixed_matrix_g0} is unique.
Otherwise, $ p $ is not unique.

We consider the equation
\begin{equation}\label{tilde_u}
	\begin{split}
		(\mcA + \mcB\mcU\mcB^T ) \tilde u =  \mcF.
	\end{split}
\end{equation}
As the subspace $ \mbC_0 $ vanishes,
this equation is well-posed 
and has a unique solution.
We divide $ \tilde u $ into two $ \mcM $-orthogonal parts
\begin{equation}\label{tilde_u_decomp}
	\begin{split}
		\tilde u \triangleq  \tilde u^{(1)} + \tilde u^{(2)},
	\end{split}
\end{equation} 
where $ \tilde u^{(1)} \in \mbC_1 $ and $ \tilde u^{(2)}\in\mbC_2 $.
Putting the decompositions \eqref{F_decomp} and \eqref{tilde_u_decomp}  
into the equation \eqref{tilde_u},
we have
\begin{equation}\nonumber %label{}
	\begin{split}
		(\mcA + \mcB\mcU\mcB^T ) \tilde u &= \mcF, \\
		(\mcA + \mcB\mcU\mcB^T ) \left( \tilde u^{(1)} + \tilde u^{(2)} \right)  
		&= \mcM f^{(1)} + \mcM f^{(2)}, \\
		\underbrace{\mcA  \tilde u^{(1)} - \mcM f^{(1)} }_{\mcM\mbC_1} 
		&= \underbrace{ \mcB\mcU\mcB^T \tilde u^{(2)}  - \mcM f^{(2)} }_{\mcM\mbC_2}.
	\end{split}
\end{equation}  
By Theorem \ref{invariant},
we know that 
the two parts in the equation above belong to 
two $ \mcM^{-1} $-orthogonal subspaces, respectively.
To make this equation hold,
both parts should be zero.
By Theorem \ref{unique_solution},
$ \tilde u^{(1)} \in \mbC_1 $ and $ \tilde u^{(2)}\in\mbC_2 $
are the unique solutions of 
$ \mcA  \tilde u^{(1)} = \mcM f^{(1)} $  
and  $ \mcB\mcU\mcB^T \tilde u^{(2)} = \mcM f^{(2)} $, respectively.    
Then if we have the solution $ \tilde u $ of the equation \eqref{tilde_u},
the components in the decomposition
\eqref{F_decomp_C0} for $ \mcF $
can be obtained explicitly:
\begin{equation}\label{F_tilde_u}
	\begin{split}
		\mcM f^{(1)} & \equiv \mcA \tilde u^{(1)} 
		= \mcA \left( \tilde u^{(1)} + \tilde u^{(2)} \right) 
		\equiv  \mcA \tilde u, \\
		\mcM f^{(2)} & \equiv \mcB\mcU\mcB^T \tilde u^{(2)} 
    	= \mcB\mcU\mcB^T \left( \tilde u^{(1)} + \tilde u^{(2)} \right) 
		\equiv \mcB\mcU\mcB^T \tilde u.
	\end{split}
\end{equation}    

\begin{remark}
	If we have the solution $ \tilde u $ of the equation \eqref{tilde_u},
	let $ p = \mcU\mcB^T \tilde u $.
	By the explicit decomposition \eqref{F_tilde_u},
	this $ p $ satisfies $ \mcB p = \mcM f^{(2)} $ in \eqref{p_Mf2}.
	Then we obtain  a solution of $ p $ for the system \eqref{mixed_matrix_g0}.
\end{remark}

Next, we consider the equation
\begin{equation}\label{bar_u}
	\begin{split}
		(\mcA + \mcB\mcU\mcB^T + c\mcM) \bar u = \mcF - \mcB\mcU\mcB^T \tilde u.
	\end{split}
\end{equation}
We divide $ \bar u $ into two $ \mcM $-orthogonal parts
\begin{equation}\label{bar_u_decomp}
	\begin{split}
		\bar u \triangleq  \bar u^{(1)} + \bar u^{(2)},
	\end{split}
\end{equation} 
where $ \bar u^{(1)} \in \mbC_1 $ and $ \bar u^{(2)}\in\mbC_2 $.
Putting the decompositions \eqref{F_decomp} and \eqref{bar_u_decomp} 
into the equation \eqref{bar_u},
using the explicit decomposition \eqref{F_tilde_u} for $ \mcF $,
we have
\begin{equation}\label{bar_u_equ_decomp}
	\begin{split}
		(\mcA + \mcB\mcU\mcB^T + c\mcM) \bar u &= \mcF - \mcB\mcU\mcB^T \tilde u ,\\
		(\mcA + \mcB\mcU\mcB^T + c\mcM) \left( \bar u^{(1)} + \bar u^{(2)}\right)  
		&= \mcM f^{(1)}, \\
		\underbrace{ (\mcA + c\mcM) \bar u^{(1)} - \mcM f^{(1)} }_{\mcM\mbC_1}   
		&= \underbrace{ - (\mcB\mcU\mcB^T + c\mcM) \bar u^{(2)}}_{\mcM\mbC_2}.
	\end{split}
\end{equation}
By Theorem \ref{unique_solution},
the $ \bar u^{(1)} \in \mbC_1 $  is the unique solution of
\begin{equation}\label{bar_u_1_component}
	\begin{split}
		(\mcA + c\mcM) \bar u^{(1)} = \mcM f^{(1)}
	\end{split}
\end{equation}
and the $ \bar u^{(2)} = 0 \in \mbC_2 $  is the unique solution of
$ (\mcB\mcU\mcB^T + c\mcM) \bar u^{(2)} = 0 $.
Then combining with the decomposition \eqref{bar_u_decomp},
we obtain the unique solution $ \bar u $ of the equation \eqref{bar_u}.

If we compare the solution $ \bar u $  of \eqref{bar_u_equ_decomp}
and the solution $ u $ of \eqref{u_system},
the only component $ u^{(1)} $ and $ \bar u^{(1)} $ contained in them
satisfy the same equation \eqref{u_1_component} or \eqref{bar_u_1_component}.
Then we obtain that
\begin{equation}\nonumber%\label{}
	\begin{split}
		\bar u = u.
	\end{split}
\end{equation}
Consequently,
the solution of $ u\in \mbC^N $ in the system \eqref{mixed_matrix_g0}
can be obtained through 
solving the two equations \eqref{tilde_u} and \eqref{bar_u}.
We summarize the two equations as 
an equivalent problem for
the solution $ u\in\mbC^N $ in the system \eqref{mixed_matrix_g0}:
\begin{equation}\label{equiv_H0_all_g0}
	\boxed{
		\begin{split}
			&\text{If  $ \dim \mbC_0 = 0 $,
				solve $ \tilde u \in\mbC^N $ such that }\\
			&\qqq (\mcA + \mcB\mcU\mcB^T) \tilde u  = \mcF ,\\
			&\text{then solve $ u \in \mbC^N $ such that}\\
			&\qq (\mcA + \mcB\mcU\mcB^T + c\mcM) u 
			= \mcF - \mcB\mcU\mcB^T \tilde u.\\
		\end{split}
	}
\end{equation}

For the case $ \dim \mbC_0 = 0 $ and $ c > 0 $, 
we consider the equation
\begin{equation}\label{hat_u}
	\begin{split}
		(\mcA + c\mcM) \hat u = \mcF - \mcB\mcU\mcB^T \tilde u.
	\end{split}
\end{equation}
%By the Definition \xx, this equation is also well-posed.
We divide $ \hat u $ into two $ \mcM $-orthogonal parts
\begin{equation}\label{hat_u_decomp}
	\begin{split}
		\hat u \triangleq  \hat u^{(1)} + \hat u^{(2)},
	\end{split}
\end{equation} 
where $ \hat u^{(1)} \in \mbC_1 $ and $ \hat u^{(2)}\in\mbC_2 $.
Putting the decompositions \eqref{F_decomp} and \eqref{hat_u_decomp} 
into the equation \eqref{hat_u},
we have
\begin{equation}\nonumber %label{}
	\begin{split}
		(\mcA  + c\mcM) \hat u &= \mcF - \mcB\mcU\mcB^T \tilde u,\\
		(\mcA   + c\mcM) \left( \hat u^{(1)} + \hat u^{(2)}\right)  
		&= \mcM f^{(1)}, \\
		\underbrace{ (\mcA + c\mcM) \hat u^{(1)} - \mcM f^{(1)} }_{\mcM\mbC_1}   
		&= \underbrace{ -  c\mcM  \hat u^{(2)}}_{\mcM\mbC_2}.
	\end{split}
\end{equation}
By Theorem \ref{unique_solution},
the $ \hat u^{(1)}\in \mbC_1 $  is the unique solution of
$ (\mcA + c\mcM) \hat u^{(1)} = \mcM f^{(1)} $.
As $ \mcM $ is full-rank.
the $ \hat u^{(2)} =0 \in \mbC_2 $ is the unique solution of
$ c\mcM  \hat u^{(2)} =0 $.
If we compare the component in the solution $ \hat u $ of \eqref{hat_u_equ_decom} 
and the solution $ u $ of \eqref{u_system},
we can find that
\begin{equation}\label{hat_u_equ_decom}
	\begin{split}
		\hat u = u.
	\end{split}
\end{equation}
Then we have another equivalent problem for 
the solution $ u\in\mbC^N $ in the system \eqref{mixed_matrix_g0}:
\begin{equation}\label{equiv_H0_c1_g0}
	\boxed{
		\begin{split}
			&\text{If  $ \dim \mbC_0 = 0 $ and $ c>0 $,}\\
			&\text{	solve $ \tilde u \in\mbC^N $ such that }\\
			&\qqq (\mcA + \mcB\mcU\mcB^T) \tilde u  = \mcF ,\\
			&\text{ then solve $ u \in \mbC^N $ such that}\\
			&\qq (\mcA + c\mcM) u 
			= \mcF - \mcB\mcU\mcB^T \tilde u.\\
		\end{split}
	}
\end{equation}

\begin{remark}
	If we use direct method,
	the equation \eqref{hat_u} is a little easier to solve
	than the equation \eqref{bar_u},
	since the equation \eqref{bar_u} has more nonzero entries 
	because of the term $ \mcB \mcU \mcB^T $.
	When using iterative methods,
	the equation \eqref{bar_u} may be better than the equation \eqref{hat_u}, 
	especially for large-scalar problems,
	because the spectral distribution of the equation \eqref{bar_u} is Laplace-like,
	which we have discussed in our paper \cite{Aux_iter}.
\end{remark}

When $ c = 0 $ in the case $ \dim \mbC_0 = 0 $,
we can construct another equivalent problem for the system \eqref{mixed_matrix_g0}.
By Theorem \ref{unique_solution},
let $ u^{(1)} $ be the unique solution in the subspace $ \mbC_1 $
of the equation $ \mcA  u^{(1)} = \mcM f^{(1)} $
and $ u^{(2)}  $ be the unique solution in the subspace $ \mbC_2 $
of the equation $ \mcB\mcU\mcB^T  u^{(2)} = \mcM f^{(2)} $.
The $ u = u^{(1)} $ is the unique solution in the system 
\eqref{mixed_matrix_g0} in the case $ \dim\mbC_0 = 0 $ and $ c=0 $.
Then for a positive number $ \alpha_1 >0 $,
we can verify that
\begin{equation}\label{u_1_H0}
	\begin{split}
		u_1 \triangleq  u^{(1)} + \frac{1}{\alpha_1} u^{(2)}\\
	\end{split}
\end{equation}
is the unique solution of the equation
\begin{equation}\label{u_1_equ}
	\begin{split}
		(\mcA + \alpha_1 \mcB\mcU\mcB^T ) u_1 &=  \mcF, \\
		(\mcA + \alpha_1 \mcB\mcU\mcB^T ) (u^{(1)} + \frac{1}{\alpha_1} u^{(2)}) 
		&=  \mcM f^{(1)} + \mcM f^{(2)},\\
		\underbrace{\mcA  u^{(1)} - \mcM f^{(1)} }_{\mcM\mbC_1} 
		&= \underbrace{\mcM f^{(2)} - \mcB\mcU\mcB^T  u^{(2)} }_{\mcM\mbC_2}.
	\end{split}
\end{equation}
Similarly, for another positive number $ \alpha_2 \not= \alpha_1 $,
we know that
\begin{equation}\label{u_2_H0}
	\begin{split}
		u_2 \triangleq  u^{(1)} + \frac{1}{\alpha_2} u^{(2)}
	\end{split}
\end{equation} 
is the solution of the equation
\begin{equation}\label{u_2_equ}
	\begin{split}
		(\mcA + \alpha_2 \mcB\mcU\mcB^T ) u_2 = \mcF.
	\end{split}
\end{equation}
By combining $ u_1 $ in \eqref{u_1_H0} and $ u_2 $ in \eqref{u_2_H0}, 
we can obtain the solution $ u\in \mbC^N $ in the system \eqref{mixed_matrix_g0} by
\begin{equation}\nonumber %label{}
	\begin{split}
		u \equiv u^{(1)} =\frac{\alpha_1 u_1 - \alpha_2 u_2}{\alpha_1 - \alpha_2}.
	\end{split}
\end{equation}
Then in the case $ \dim \mbC_0 = 0 $ and $ c = 0 $,
the solution $ u\in \mbC^N $ in the system \eqref{mixed_matrix_g0}
can be obtained through the two equations \eqref{u_1_equ} and \eqref{u_2_equ}:
\begin{equation}\label{equiv_H0_c0_g0}
	\boxed{
		\begin{split}
			&\text{If $ \dim \mbC_0 = 0 $ and $ c = 0 $,} \\
			& \text{ take two nonzero number 
				$ \alpha_1 \not = \alpha_2 $ }\\
			& \text{ and solve $ u_1, u_2 \in\mbC^N $ such that}\\
			&\qqq \left( \mcA  + \alpha_1 \mcB^T \mcU \mcB \right) u_1 = \mcF \qq \text{and}\\
			&\qqq \left( \mcA  + \alpha_2 \mcB^T \mcU \mcB \right) u_2 = \mcF.\\
			& \text{ Then the solution is}\\
			&\qqq \qqq u = \frac{\alpha_1 u_1 - \alpha_2 u_2 }{\alpha_1 - \alpha_2}.
		\end{split}
	}
\end{equation}

\begin{remark}	
	The equivalent problems \eqref{equiv_H0_all_g0} and \eqref{equiv_H0_c0_g0}
	are mathematically equivalent in the case $ \dim \mbC_0 = 0 $ and $ c = 0 $.
	The difference is that 
	the last equation in \eqref{equiv_H0_all_g0}
	depends on the first solution,
	while the two equations in \eqref{equiv_H0_c0_g0} are independent.
	When using iterative methods,
	two equations in \eqref{equiv_H0_c0_g0}
	can be solved simultaneously.
\end{remark}

\subsection{The case $ \dim \mbC_0 \not = 0 $ and $ c>0 $}

In this case, the subspace $ \mbC_0 $ is not trivial.
According to the decomposition \eqref{decomp_M},
the vector $ \mcM^{-1}\mcF $ can be divided into
three $\mcM $-orthogonal parts
\begin{equation}\nonumber%\label{}
	\begin{split}
		\mcM^{-1} \mcF \triangleq f^{(0)}  + f^{(1)} + f^{(2)},
	\end{split}
\end{equation} 
where $ f^{(0)} \in \mbC_0 $, $ f^{(1)}\in \mbC_1 $ and $ f^{(2)}\in \mbC_2 $.
Then there is a $\mcM^{-1} $-orthogonal decomposition for $ \mcF $
\begin{equation}\label{F_decomp_c0}
	\begin{split}
		\mcF = \mcM f^{(0)}  + \mcM f^{(1)} + \mcM f^{(2)},
	\end{split}
\end{equation} 
where $ \mcM f^{(0)}\in \mcM\mbC_0 $, $ \mcM f^{(1)} \in \mcM\mbC_1 $ 
and $ \mcM f^{(2)}\in \mcM \mbC_2 $.

From the constraint condition $  \mcB^T u = 0 $ in the system \eqref{mixed_matrix_g0}
and by Theorem \ref{kerBandkerBB},
we know that
\begin{equation}\label{u_decomp_c0}
	\begin{split}
		u \in \Ker \mcB^T = \mbC_0 \oplus_\mcM \mbC_1.
	\end{split}
\end{equation} 
Then the solution of $ u \in \mbC^N $ in the system \eqref{mixed_matrix_g0}
can be divided into 
two $\mcM $-orthogonal parts
\begin{equation}\label{u_decomp_H1}
	\begin{split}
		u \triangleq u^{(0)}  + u^{(1)},
	\end{split}
\end{equation} 
where $ u^{(0)} \in \mbC_0 $ and $ u^{(1)}\in \mbC_1 $.
Putting the decompositions \eqref{F_decomp_c0} and \eqref{u_decomp_H1} into 
the first equation of the system \eqref{mixed_matrix_g0},
we have
\begin{equation}\label{u_p_decomp}
	\begin{split}
		(\mcA + c\mcM)u + \mcB p &= \mcF, \\
		(\mcA + c\mcM)(u^{(0)}  + u^{(1)}) + \mcB p &= \mcM{f^{(0)}  + \mcM f^{(1)} + \mcM f^{(2)}}, \\
		\underbrace{c \mcM u^{(0)} - \mcM f^{(0)}}_{\mcM\mbC_0} 
		+ \underbrace{(\mcA + c\mcM) u^{(1)} - \mcM f^{(1)}}_{\mcM\mbC_1} 
		&= \underbrace{\mcM f^{(2)} - \mcB p}_ {\mcM\mbC_2}.
	\end{split}
\end{equation}
By Theorem \ref{invariant} 
and $ \mcB p \in \IM \mcB = \mcM \mbC_2 $ in Theorem \ref{ImB},
we know that 
the three parts in the equation above belong to 
three $ \mcM^{-1} $-orthogonal subspaces, respectively.
To make this equation hold,
all the three parts should be zero.
By Theorem \ref{unique_solution},
there is a unique solution $ u^{(1)} \in \mbC_1 $ satisfying
\begin{equation}\nonumber %label{}
	\begin{split}
		(\mcA + c\mcM) u^{(1)} = \mcM f^{(1)}.
	\end{split}
\end{equation}
As $ \mcM $ is full-rank,
\begin{equation}\nonumber %label{}
	\begin{split}
		u^{(2)} = \frac{1}{c}f^{(2)}
	\end{split}
\end{equation}
is the unique solution of the equation
\begin{equation}\nonumber %label{}
	\begin{split}
		c \mcM u^{(0)} - \mcM f^{(0)} = 0.
	\end{split}
\end{equation}
Then by the decomposition \eqref{u_decomp_H1},
the solution of $ u\in \mbC^N $ in the system \eqref{mixed_matrix_g0} 
is unique
in the case $ \dim \mbC_0 \not= 0 $ and $ c>0 $.
By Theorem \ref{ImB},
there exists a $ p\in \mbC^M $ such that
\begin{equation}\label{p_Mf3}
	\begin{split}
		\mcB p = \mcM f^{(2)}.
	\end{split}
\end{equation}

We compute the eigenvectors of the zero eigenvalue of the 
generalized eigenvalue problem
\begin{equation}\label{AB_eig_c0}
	\begin{split}
		(\mcA + \mcB\mcU\mcB^T) u = \lambda \mcM u.
	\end{split}
\end{equation}	
The eigenvectors of the zero eigenvalue can form a base 
$ \left\lbrace \mfu^{(0)}_1 \right\rbrace _{i=1}^{\dim \mbC_0} $ 
of the subspace $ \mbC_0 $
that satisfies Definition \ref{H_def}.
Let $ \mcH \in \mbC^{ N \times \dim \mbC_0} $:
\begin{equation}\nonumber %label{}
	\begin{split}
		\mcH = 
		\left[ \mfu^{(0)}_1,\mfu^{(0)}_2,\cdots, \mfu^{(0)}_{\dim \mbC_0} \right].
	\end{split}
\end{equation}
We consider the following equation
\begin{equation}\label{tilde_u_c0}
	\begin{split}
		\left( \mcA + \mcB\mcU\mcB^T + \mcM \mcH \mcH^T \mcM\right)
		\tilde u = \mcF.
	\end{split}
\end{equation}		
We divide $ \tilde u $ into three $ \mcM $-orthogonal parts
\begin{equation}\label{tilde_u_decomp_c0} 
	\begin{split}
		\tilde u \triangleq  \tilde u^{(0)}  + \tilde u^{(1)} + \tilde u^{(2)},
	\end{split}
\end{equation}
where $ \tilde u^{(0)} \in \mbC_0 $, $ \tilde u^{(1)}\in \mbC_1 $ 
and $ \tilde u^{(2)}\in \mbC_2 $.
Putting the decompositions \eqref{F_decomp_c0} and \eqref{tilde_u_decomp_c0} 
into the equation \eqref{tilde_u_c0},
we have 
\begin{equation}\nonumber %label{}
	\begin{split}
		\left( \mcA + \mcB\mcU\mcB^T + \mcM \mcH \mcH^T \mcM\right)\tilde u 
		&= \mcF\\
		\left( \mcA + \mcB\mcU\mcB^T + \mcM \mcH \mcH^T \mcM\right)
		\left( \tilde u^{(0)}  + \tilde u^{(1)} + \tilde u^{(2)}\right)  
		&=  \mcM f^{(0)} + \mcM f^{(1)} + \mcM f^{(2)}  \\
		\mcA \tilde u^{(1)} + \mcB\mcU\mcB^T \tilde u^{(2)} + \mcM \mcH \mcH^T \mcM \tilde u^{(0)}
		&=   \mcM f^{(0)} + \mcM f^{(1)} + \mcM f^{(2)} \\
		\underbrace{\mcM \mcH \mcH^T \mcM \tilde u^{(0)} - \mcM f^{(0)} }_{\mcM\mbC_0}
		+ \underbrace{\mcA \tilde u^{(1)} - \mcM f^{(1)} }_{\mcM\mbC_1} 
		&= \underbrace{ \mcB\mcU\mcB^T \tilde u^{(2)}  - \mcM f^{(2)} }_{\mcM\mbC_2}.
	\end{split}
\end{equation}	
By Theorem \ref{invariant},
we know that 
the three parts in the equation above belong to 
three $ \mcM^{-1} $-orthogonal subspaces, respectively.
To make this equation hold,
all the three parts should be zero.
By Theorem \ref{unique_solution},
$ \tilde u^{(0)} \in \mbC_0 $ is the unique solution of 
$ \mcM \mcH \mcH^T \mcM \tilde u^{(0)} = \mcM f^{(0)} $,
$ \tilde u^{(1)} \in \mbC_1 $ is the unique solution of 
$ \mcA \tilde u^{(1)} = \mcM f^{(1)} $  
and $ u^{(2)}  \in \mbC_2 $ is the unique solution of 
$ \mcB\mcU\mcB^T \tilde u^{(2)}  = \mcM f^{(2)} $.
Similar to \eqref{F_tilde_u}, 
if we have the solution $ \tilde u $ of the equation \eqref{tilde_u_c0},
we can obtain the explicit decomposition for 
the right hand side $ \mcF $ in the case $ \dim \mbC_0 \not = 0 $:
\begin{equation}\label{F_tilde_u_H1}
	\begin{split}
		\mcM f^{(0)} &= \mcM \mcH \mcH^T \mcM \tilde u,\\
		\mcM f^{(1)} &= \mcA \tilde u, \\
		\mcM f^{(2)} &= \mcB\mcU\mcB^T \tilde u.
	\end{split}
\end{equation} 

We consider the equation
\begin{equation}\label{bar_u_c0}
	\begin{split}
		\left( \mcA + \mcB\mcU\mcB^T + c\mcM \right)  \bar u = \mcF - \mcB\mcU\mcB^T \tilde u.
	\end{split}
\end{equation} 
In this equation, 
by the decomposition \eqref{F_tilde_u_H1},
the component in $ \mcM \mbC_2 $ of $ \mcF $  is eliminated.
To study its solution, 
we also divide $ \bar u $ in three $ \mcM $-orthogonal parts
\begin{equation}\nonumber%\label{}
	\begin{split}
		\bar u \triangleq  \bar u^{(0)}  + \bar u^{(1)} + \bar u^{(2)},
	\end{split}
\end{equation} 
where $ \bar u^{(0)} \in \mbC_0 $, $ \bar u^{(1)}\in \mbC_1 $ 
and $ \bar u^{(2)}\in \mbC_2 $.
Then we have
\begin{equation}\nonumber %label{}
	\begin{split}
		(\mcA + \mcB\mcU\mcB^T + c\mcM) \hat u &=  \mcF - \mcB\mcU\mcB^T \tilde u \\
		(\mcA + \mcB\mcU\mcB^T + c\mcM) \left( \bar u^{(0)}  + \bar u^{(1)} + \bar u^{(2)}\right)  
		&=  \mcM f^{(0)} + \mcM f^{(1)} \\
		\underbrace{ c\mcM \bar u^{(0)} - \mcM f^{(0)} }_{\mcM \mbC_0}
		+ \underbrace{ (\mcA + c\mcM) \bar u^{(1)} - \mcM f^{(1)} }_{\mcM \mbC_1}
		&= - \underbrace{(\mcB\mcU\mcB^T + c\mcM) \bar u^{(2)} }_{\mcM \mbC_2}.
	\end{split}
\end{equation}
To make this equation hold,
all the three parts should be zero as before.
By Theorem \ref{unique_solution},
$ \bar u^{(1)} \in \mbC_1 $ is the unique solution of the equation
\begin{equation}\nonumber %label{}
	\begin{split}
		(\mcA + c\mcM) \bar u^{(1)} = \mcM f^{(1)},
	\end{split}
\end{equation}
and  $ \bar u^{(1)} = 0 \in \mbC_2 $ is the unique solution of the equation
\begin{equation}\nonumber %label{}
	\begin{split}
		(\mcB\mcU\mcB^T + c\mcM) \bar u^{(2)}= 0.
	\end{split}
\end{equation}
As $ \mcM $ is full-rank,
\begin{equation}\nonumber %label{}
	\begin{split}
		\bar u^{(0)} = \frac{1}{c}f^{(0)}
	\end{split}
\end{equation}
is the unique solution of the equation
\begin{equation}\nonumber %label{}
	\begin{split}
		c \mcM \hat u^{(0)} - \mcM f^{(0)} = 0.
	\end{split}
\end{equation}

If we compare each component in the solution $ \bar u $ of the equation \eqref{bar_u_c0} 
and the solution $ u $ of the equation \eqref{u_p_decomp},
we obtain that
\begin{equation}\nonumber%\label{}
	\begin{split}
		\bar u = u.
	\end{split}
\end{equation}
Then we have an equivalent problem for 
the solution $ u\in \mbC^N $ in the system \eqref{mixed_matrix_g0}
in the case $ \dim \mbC_0 \not= 0 $ and $ c>0 $:	
\begin{equation}\label{equiv_H1_all_g0}
	\boxed{
		\begin{split}
			& \text{If $ \dim \mbC_0 \not = 0 $ and $ c > 0 $,} \\ 
			& \text{ find the eigenvectors $ \mcH $ with zero eigenvalue}\\
			& \qqq \left( \mcA + \mcB\mcU\mcB^T\right)  u = \lambda \mcM u,\\
			&\text{ then	solve $ \tilde u \in\mbC^N $ such that }\\
			&\qqq \left( \mcA + \mcB\mcU\mcB^T + \mcM \mcH \mcH^T \mcM\right)  \tilde u = \mcF,\\
			&\text{ then solve $ u \in \mbC^N $ such that}\\
			&\qq \left( \mcA  + \mcB\mcU\mcB^T + c\mcM\right)  u = \mcF - \mcB\mcU\mcB^T \tilde u.\\
		\end{split}
	}
\end{equation}		

\begin{remark}
	
	The subspace $ \mbC_0 $ is an inherent characteristic of a system.
	For a system, the eigenvalue problems \eqref{AB_eig_c0} 
	needs to be computed only once.
	The matrix $ \mcH $ can be fixed when the right hand side varies.
	
	As the eigenvectors with zero eigenvalue of 
	the eigenvalue problem
	are usually full vectors,
	the term $ \mcM \mcH \mcH^T \mcM $ is probably a full matrix.
	This results that if the dimension of the matrix is a litter large,
	direct methods become inefficient and even impossible
	when computing the equation \eqref{tilde_u_c0}.
	An alternative choice is to use iterative methods,
	where the matrix-vector multiplication $ \mcM \mcH \mcH^T \mcM u $
	can be computed one by one.
	Thus we can avoid dealing with the term $ \mcM \mcH \mcH^T \mcM $ in an explicit way.
\end{remark}

Similar to \eqref{hat_u}, as $ c>0 $,
we consider the equation
\begin{equation}\label{hat_u_c0}
	\begin{split}
		(\mcA + c\mcM) \hat u = & \mcF - \mcB\mcU\mcB^T \tilde u.
	\end{split}
\end{equation}
We also divide $ \hat u \in \mbC^N $ into three $ \mcM $-orthogonal parts
\begin{equation}\label{hat_u_decomp_c0}
	\begin{split}
		\hat u \triangleq  \hat u^{(0)}  + \hat u^{(1)} + \hat u^{(2)},
	\end{split}
\end{equation} 
where $ \hat u^{(0)} \in \mbC_0 $, $ \hat u^{(1)}\in \mbC_1 $ 
and $ \hat u^{(2)}\in \mbC_2 $.
Putting the decompositions \eqref{F_decomp_c0} and \eqref{hat_u_decomp_c0} 
into the equation \eqref{hat_u_c0},
we have 
\begin{equation}\nonumber %label{}
	\begin{split}
		(\mcA + c\mcM) \hat u &=  \mcF - \mcB\mcU\mcB^T \tilde u, \\
		(\mcA + c\mcM) \left( \hat u^{(0)}  + \hat u^{(1)} + \hat u^{(2)}\right)  
		&= \mcM f^{(0)} + \mcM f^{(1)}, \\
		\underbrace{ c\mcM \hat u^{(0)} - \mcM f^{(0)} }_{\mcM \mbC_0}
		+ \underbrace{ (\mcA + c\mcM) \hat u^{(1)} - \mcM f^{(1)} }_{\mcM \mbC_1}
		&= - \underbrace{c\mcM\hat u^{(2)} }_{\mcM \mbC_2}.
	\end{split}
\end{equation}
To make this equation hold,
all the three parts should be zero.
By Theorem \ref{unique_solution},
$ \hat u^{(1)} \in \mbC_1 $ is the unique solution of the equation
$ \left( \mcA + c\mcM\right)  \hat u^{(1)} = \mcM f^{(1)} $.
As $ \mcM $ is full-rank,
$ \hat u^{(0)} = \frac{1}{c}f^{(0)} $
is the unique solution of the equation
$ c \mcM \hat u^{(0)} - \mcM f^{(0)} = 0 $
and $ \hat u^{(1)} = 0 \in \mbC_2 $ is the unique solution of the equation
$ c\mcM\hat u^{(2)} = 0 $.
If we compare each component in the solution $ \hat u $ of the equation \eqref{hat_u_c0} 
and the solution $ u $ of the equation \eqref{u_p_decomp},
we obtain that
\begin{equation}\nonumber%\label{}
	\begin{split}
		\hat u = u.
	\end{split}
\end{equation}
Then we have another equivalent problem for 
the solution $ u\in \mbC^N $ in the system \eqref{mixed_matrix_g0}
in the case $ \dim \mbC_0 \not= 0 $ and $ c>0 $:	
\begin{equation}\label{equiv_H1_2_g0}
	\boxed{
		\begin{split}
			& \text{If $ \dim \mbC_0 \not = 0 $ and $ c > 0 $,} \\ 
			& \text{ find the eigenvectors $ \mcH $  zero eigenvalue}\\
			& \qqq \left( \mcA + \mcB\mcU\mcB^T\right)  u = \lambda \mcM u,\\
			&\text{ then	solve $ \tilde u \in\mbC^N $ such that }\\
			&\qqq \left( \mcA + \mcB\mcU\mcB^T + \mcM \mcH \mcH^T \mcM\right)  \tilde u = \mcF,\\
			&\text{ then solve $ u \in \mbC^N $ such that}\\
			&\qq \left( \mcA   + c\mcM\right)  u = \mcF - \mcB\mcU\mcB^T \tilde u.\\
		\end{split}
	}
\end{equation}

\subsection{The case $ \dim \mbC_0 \not= 0 $ and $ c = 0 $} 
In the case $ c = 0 $,
the problem \eqref{mixed_matrix_g0}  becomes
\begin{equation}\label{mixed_matrix_g0_c0}
	\begin{split}
		\mcA u + \mcB p &= \mcF, \\
		\mcB^T u &= 0.
	\end{split}
\end{equation}
% In the case $ \dim \mbC_0 \not= 0 $,
% we have the $ \mcM $-orthogonal decomposition for $ \mcM^{-1}\mcF $
% and $ \mcM^{-1} $-orthogonal decompositions for 
% $ \mcM^{-1}\mcF $ and $ \mcF $, respectively:
% \begin{equation}\label{F_decomp_c00}
% 	\begin{split}
% 		%
% 		\mcM^{-1} \mcF &\triangleq f^{(0)}  + f^{(1)} + f^{(2)},\\
% 		\mcF &= \mcM f^{(0)}  + \mcM f^{(1)} + \mcM f^{(2)}.
% 		%
% 	\end{split}
% \end{equation} 
% Here $ f^{(0)} \in \mbC_0 $, $ f^{(1)}\in \mbC_1 $ and $ f^{(2)}\in \mbC_2 $.
% To satisfy  the constraint condition $ \mcB^T u = 0 $ 
% in \eqref{mixed_matrix_g0},
% we know that
% \begin{equation}\label{u_decomp_c00}
% 	\begin{split}
% 		%
% 		u \in \Ker \mcB^T = \mbC_0 \oplus_\mcM \mbC_1.
% 		%
% 	\end{split}
% \end{equation} 
% Then $ u $ can be decomposed into 
% two $\mcM^{-1} $-orthogonal parts:
% \begin{equation}\nonumber%\label{}
% 	\begin{split}
% 		%
% 		u \triangleq u^{(0)}  + u^{(1)},
% 		%
% 	\end{split}
% \end{equation} 
% where $ u^{(0)} \in \mbC_0 $ and $ u^{(1)}\in \mbC_1 $.
% Putting the decompositions \eqref{F_decomp_c00}
% and \eqref{u_decomp_c00} into the first equation 
% of \eqref{mixed_matrix_g0_c0}, 
% we have
As $ \dim \mbC_0 \not= 0 $,
the right-hand side $ \mcF \in \mbC^N $ 
and the solution $ u\in \mbC^N $ 
in this system
can be still decomposed as 
\eqref{F_decomp_c0} and \eqref{u_decomp_H1}, respectively.
Putting the two decompositions
into the first equation of the system \eqref{mixed_matrix_g0_c0},
we have
\begin{equation}\nonumber %label{}
	\begin{split}
		\mcA u + \mcB p &= \mcF, \\
		\mcA (u^{(0)}  + u^{(1)}) + \mcB p 
		&= \mcM f^{(0)}  + \mcM f^{(1)} + \mcM f^{(2)}, \\
		\underbrace{ \left( \mcA u^{(1)} - \mcM f^{(1)} \right) }_{\mcM\mbC_1} 
		+ \underbrace{ \left( \mcB p - \mcM f^{(2)} \right) }_{\mcM\mbC_2} 
		&= \underbrace{ \mcM f^{(0)} }_{\mcM\mbC_0}.
	\end{split}
\end{equation}
As the three parts in this equation are $ \mcM^{-1} $-orthogonal,
if the term $ \mcM f^{(0)} \not = 0 $,
the equation above never hold.
If $ \mcM f^{(0)} = 0 $,
there exist $ u^{(1)} $ and $ p $ that make this equation hold.
Here,
$ u^{(1)} $ is the unique solution of 
$ \mcA u^{(1)} = \mcM f^{(1)} $ in $ \mcM \mbC_1 $ 
by Theorem \ref{unique_solution}
and $ p $ satisfies $ \mcB p = \mcM f^{(2)} $ by Theorem \ref{ImB}.
However, there is no restriction on 
the component $ u^{(0)} $.
For any $ u^{(0)} \in \mbC_0 $,
$ u = u^{(0)}  + u^{(1)} $ is a solution for $ u\in \mbC^N $ 
in the system \eqref{mixed_matrix_g0_c0}
when $ \mcM f^{(0)} = 0 $.

In this paper, 
we shall not consider this case.

\subsection{Summary}
We summarize the equivalent problems in Table \ref{table_g0}
for the solution $ u \in \mbC^N $
in the system \eqref{mixed_matrix_g0} 
in the case $ \mcG= 0 $.
\begin{table}[ht]  %\label{ }	
	\centering  
	\begin{tabular}{| c| c | c |}
		\hline
		& $ c>0 $ & $ c=0 $ \\
		\hline
		$ \dim \mbC_0  = 0 $ 
		&  \eqref{equiv_H0_all_g0} \eqref{equiv_H0_c1_g0}   
		& \eqref{equiv_H0_all_g0} \eqref{equiv_H0_c0_g0} \\
		\hline
		$ \dim \mbC_0 \not = 0 $ 
		& \eqref{equiv_H1_all_g0} \eqref{equiv_H1_2_g0} 
		& \tabincell{c}{no solution or \\ no unique solution}\\	
		\hline	
	\end{tabular}
	\caption{The equivalent problems for the solution $ u \in \mbC^N $ 
		in the system \eqref{mixed_matrix_g0} in the case $ \mcG= 0 $.}
	\label{table_g0}
\end{table}

%\clearpage

\section{ The case $ \mcG \not = 0 $}\label{sec_4}

In this section,
we consider the general case $ \mcG \not = 0 $ for the problem
\begin{equation}\label{mixed_matrix_g1}
	\begin{split}
		(\mcA + c\mcM)u + \mcB p &= \mcF, \\
		\mcB^T u &= \mcG.
	\end{split}
\end{equation} 
In this case, 
$ \mcG $ is required to be contained in $ \IM \mcB^T $,
i.e. $ \mcG \in \IM\mcB^T $.
Otherwise, there is no solution satisfying 
the constraint condition $ \mcB^T u = \mcG $.

\begin{remark}
	There is a special case 
	that $ \Ker \mcB $ vanishes in Theorem \ref{CM_decomp} 
	and $ \IM\mcB^T = \mbC^M $.
	Then there always exists a $ u \in \mbC^N $ such that $ \mcB^T u = \mcG $
	for any $ \mcG \in \mbC^M $.
\end{remark}

In the case $ 0 \not= \mcG \in \IM \mcB^T $,
if we have a $ u_g \in \mbC^N $ satisfying 
\begin{equation}\label{BTug}
	\begin{split}
		\mcB^T u_g = \mcG,
	\end{split}
\end{equation}
then the solution $ u\in\mbC^N $ in the system \eqref{mixed_matrix_g1}
can be the superposition of the two components
\begin{equation}\nonumber %label{}
	\begin{split}
		u = u_0 + u_g.
	\end{split}
\end{equation}
Here $ u_0 $ is the solution of the system
\begin{equation}\label{mixed_matrix_u0}
	\begin{split}
		(\mcA + c\mcM)u_0 + \mcB p &= \mcF - (\mcA + c\mcM)u_g,\\
		\mcB^T u_0 &= 0.
	\end{split}
\end{equation}
The solution of the system \eqref{mixed_matrix_u0} 
has been considered in the previous section.
The remained thing is 
how to find a $ u_g \in \mbC^N $ satisfying 
the constraint condition $ \mcB^T u_g = \mcG $.

By Theorem \ref{ImB}, 
we have
\begin{equation}\nonumber %label{}
	\begin{split}
		\mcB \mcU \mcG \in \mcM \mbC_2.
	\end{split}
\end{equation}
We consider the equation
\begin{equation}\label{u_g}
	\begin{split}
		(\mcA + \mcB\mcU\mcB^T ) u_g = \mcB \mcU \mcG.
	\end{split}
\end{equation} 
We divide $ u_g $ into three $ \mcM $-orthogonal parts:
\begin{equation}\label{u_g_decomp}
	\begin{split}
		u_g = u_g^{(0)} + u_g^{(1)} + u_g^{(2)},
	\end{split}
\end{equation}
where $ u_g^{(0)} \in \mbC_0 $, 
$ u_g^{(1)}\in \mbC_1 $ and $ u_g^{(2)}\in \mbC_2 $.
Putting the decomposition  \eqref{u_g_decomp} into the equation \eqref{u_g},
we have
\begin{equation}\nonumber %label{}
	\begin{split}
		(\mcA + \mcB\mcU\mcB^T ) u_g 
		&= \mcB \mcU \mcG,\\
		(\mcA + \mcB\mcU\mcB^T ) (u_g^{(0)} + u_g^{(1)} + u_g^{(2)}) 
		&= \mcB \mcU \mcG,\\
		\underbrace{ \mcA u_g^{(1)} }_{\mcM\mbC_1} 
		&= \underbrace{\mcB \mcU \mcG -\mcB\mcU\mcB^T u_g^{(2)} }_{\mcM\mbC_2}.
	\end{split}
\end{equation}
The two parts in the equation above are $ \mcM^{-1} $-orthogonal.
To make this equation hold,
both parts should be zero.
By Theorem \ref{unique_solution},
$ u_g^{(1)} = 0 \in \mbC_1 $ is the unique solution
of the equation $ \mcA u_g^{(1)} = 0 $
and $ u_g^{(2)} \in \mbC_1 $
is the unique solution of 
the equation
\begin{equation}\label{u_g_2}
	\begin{split}
		\mcB \mcU \mcG -\mcB\mcU\mcB^T u_g^{(2)} = 0.
	\end{split}
\end{equation}
If $ \dim \mbC_0 = 0 $, 
the component $ u_g^{(0)} $ vanishes.
The equation \eqref{u_g} is well-posed and
has a unique solution 
\begin{equation}\nonumber %label{}
	\begin{split}
		u_g = u_g^{(2)}.
	\end{split}
\end{equation}
If $ \dim \mbC_0 \not = 0 $,
the solution of the equation \eqref{u_g} is not unique.
For any $ u_g^{(0)} \in \mbC_0 $,
\begin{equation}\nonumber %label{}
	\begin{split}
		u_g = u_g^{(0)} + u_g^{(2)}
	\end{split}
\end{equation}
is a solution of the equation \eqref{u_g},
where $ u_g^{(2)} $ satisfies the equation \eqref{u_g_2}.
As $ u_g^{(0)} \in \Ker \mcB\mcU\mcB^T $ by Theorem \ref{invariant},
in the both cases, we have 
\begin{equation}\nonumber %label{}
	\begin{split}
		\mcB\mcU\mcB^T u_g = \mcB\mcU\mcB^T u_g^{(2)} = \mcB \mcU \mcG.
	\end{split}
\end{equation}
In both cases,
we know that this $ u_g $ satisfies 
the constraint condition  $ \mcB^T u_g = \mcG \in \IM \mcB^T$
by Theorem \ref{BTGU}.

In the case $ \dim \mbC_0 \not = 0 $, 
we have an alternative way to obtain a $ u_g $.
We insert the term $ \mcM\mcH\mcH^T\mcM $ into the equation \eqref{u_g} 
as that in \eqref{tilde_u_c0}
and get the equation
\begin{equation}\label{u_g_H1}
	\begin{split}
		\left( \mcA + \mcB\mcU\mcB^T + \mcM \mcH \mcH^T \mcM\right) u_g
		= \mcB \mcU \mcG.
	\end{split}
\end{equation}	
Putting the decompositions \eqref{u_g_decomp} into this equation, 
we have
\begin{equation}\nonumber %label{}
	\begin{split}
		\left( \mcA + \mcB\mcU\mcB^T + \mcM \mcH \mcH^T \mcM\right) u_g
		&= \mcB \mcU \mcG,\\
		\left( \mcA + \mcB\mcU\mcB^T + \mcM \mcH \mcH^T \mcM\right)
		\left(   u_g^{(0)}  +   u_g^{(1)} +   u_g^{(2)}\right)  
		&=  \mcB \mcU \mcG,  \\
		\mcA u_g^{(1)} + \mcB\mcU\mcB^T u_g^{(2)} + \mcM \mcH \mcH^T \mcM  u_g^{(0)}
		&=   \mcB \mcU \mcG, \\
		\underbrace{\mcM \mcH \mcH^T \mcM u_g^{(0)}  }_{\mcM\mbC_0}
		+ \underbrace{\mcA  u_g^{(1)}   }_{\mcM\mbC_1} 
		&= \underbrace{\mcB \mcU \mcG - \mcB\mcU\mcB^T u_g^{(2)}}_{\mcM\mbC_2}.
	\end{split}
\end{equation}
To make this equation hold,
$ u_g^{(0)} = 0 \in \mbC_0 $, $ u_g^{(1)} = 0 \in \mbC_1 $
and $ u_g^{(2)} \in \mbC_2 $
are the unique solutions
that make the three $ \mcM^{-1} $-orthogonal parts be zero 
by Theorem \ref{unique_solution}.
Then we obtain a unique solution of the equation \eqref{u_g_H1}
\begin{equation}\nonumber %label{}
	\begin{split}
		u_g =  u_g^{(2)} \in \mbC_2.
	\end{split}
\end{equation}
By Theorem \ref{BTGU},
this $ u_g $ also satisfies the constraint condition
$ \mcB^T u_g = \mcG \in \IM \mcB^T $.

\begin{theorem}\label{u_g_u_0}
	If $ u_g \in \mbC_2 $ such that $ \mcB^T u_g = \mcG \in \IM \mcB^T $,
	and the solution $ u $ in the system \eqref{mixed_matrix_g0}
	and the solution $ u_0 $ in the system \eqref{mixed_matrix_u0} 
	are unique,
	then we have $ u=u_0 $.
\end{theorem}

\begin{proof}
	As $ u_g \in \mbC_2 $,
	we have
	\begin{equation}\label{Mu_g}
		\begin{split}
			(\mcA + c\mcM)u_g = c \mcM u_g \in \mcM \mbC_2. 
		\end{split}
	\end{equation} 
	By the analysis in \eqref{u_system} and \eqref{u_p_decomp},
	the solutions $ u $ in \eqref{mixed_matrix_g0} and $ u_0 $ in \eqref{mixed_matrix_u0}
	are decided by the components in the subspaces $ \mcM\mbC_0 $ and $ \mcM \mbC_1 $
	of their right-hand sides, respectively.
	Compared with the system \eqref{mixed_matrix_g0}, 
	the additional term in the right-hand side of the system \eqref{mixed_matrix_u0} 
	is  $  c \mcM u_g \in \mcM \mbC_2 $ by \eqref{Mu_g}.
	This term has no influence on the solution $ u_0 \in \mbC_0 \oplus_\mcM \mbC_1 $.
	The components in $ \mbC_0 \oplus_\mcM \mbC_1 $ in the right-hand sides 
	of the two system are the same.
	Consequently, their solutions $ u $ and $ u_0 $ are equal.
\end{proof}

In the case $ \dim \mbC_0 = 0 $,
we use the equation \eqref{u_g} to solve $ u_g $,
while in the case $ \dim \mbC_0 \not = 0 $,
we use the equation \eqref{u_g_H1} to solve $ u_g $.
The $ u_g $ is unique in both equations 
and $ u_g \in \mbC_2 $ satisfying
$ \mcB^T u_g = \mcG \in \IM \mcB^T $.
Then, by Theorem \ref{u_g_u_0}, 
we can solve $ u_0 $ in the system \eqref{mixed_matrix_u0} 
by the system \eqref{mixed_matrix_g0} 
which is the case $ \mcG=0 $.
After obtaining its solution, 
we add the solution $ u_g $ of the equation \eqref{u_g} or \eqref{u_g_H1} to it.
That is the solution of the general case $ 0 \not = \mcG \in \IM \mcB^T $.

To add the term $ u_g \in \mbC_2 $,
the last equations in the equivalent problems 
\eqref{equiv_H0_all_g0}, \eqref{equiv_H0_c1_g0}, \eqref{equiv_H1_all_g0} 
and \eqref{equiv_H1_2_g0}:
\begin{equation}\nonumber%\label{}
	\begin{split}
		\left( \mcA   + c\mcM\right)  u  &= \mcF - \mcB\mcU\mcB^T \tilde u\\
		\text{and} \qq
		\left( \mcA + \mcB\mcU\mcB^T + c\mcM\right)  u &= \mcF - \mcB\mcU\mcB^T \tilde u
	\end{split}
\end{equation}
are modified as
\begin{equation}\nonumber%\label{}
	\begin{split}
		\left( \mcA   + c\mcM\right)  u  
		&= \mcF - \mcB\mcU\mcB^T \tilde u + \left( \mcA  + c\mcM\right) u_g  
		=  \mcF - \mcB\mcU\mcB^T \tilde u + c\mcM u_g  \\
		\text{and} \qq
		\left( \mcA + \mcB\mcU\mcB^T + c\mcM\right)  u 
		&= \mcF - \mcB\mcU\mcB^T \tilde u
		+ \left( \mcA + \mcB\mcU\mcB^T + c\mcM\right) u_g
		= \mcF - \mcB\mcU\mcB^T \tilde u + \mcB\mcU\mcG + c\mcM u_g.
	\end{split}
\end{equation}
Here we use $ \mcB\mcU\mcB^T u_g = \mcB\mcU\mcG $ and $ \mcA u_g = 0 $ as $ u_g \in \mbC_2 $.
Then,
based on the equivalent problems
\eqref{equiv_H0_all_g0}, \eqref{equiv_H0_c1_g0}, \eqref{equiv_H1_all_g0} 
and \eqref{equiv_H1_2_g0} for the case $ \mcG = 0 $,
the equivalent problems 
for the solution $ u\in \mbC^N $ in the system \eqref{mixed_matrix_g1}
are the follows, respectively:
\begin{equation}\label{equiv_H0_all_g1}
	\boxed{
		\begin{split}
			&\text{If  $ \dim \mbC_0 = 0 $ and $ c\geq 0 $,}\\
			&\text{	solve $ u_g \in\mbC^N $ such that }\\
			&\qqq \left( \mcA  + \mcB\mcU\mcB^T \right) u_g = \mcB \mcU \mcG,\\
			&\text{ then solve $ \tilde u \in \mbC^N $ such that}\\ 			
			&\qqq \left( \mcA  + \mcB\mcU\mcB^T \right) \tilde u  = \mcF,\\
			&\text{ then solve $ u \in \mbC^N $ such that}\\
			& \left(\mcA + \mcB\mcU\mcB^T + c\mcM \right) u 
			= \mcF - \mcB\mcU\mcB^T \tilde u + \mcB\mcU\mcG + c\mcM u_g .\\
		\end{split}
	}
\end{equation}

\begin{equation}\label{equiv_H0_c1_g1}
	\boxed{
		\begin{split}
			&\text{If  $ \dim \mbC_0 = 0 $ and $ c>0 $,}\\
			&\text{	solve $ u_g \in\mbC^N $ such that }\\
			&\qqq \left( \mcA  + \mcB\mcU\mcB^T \right) u_g = \mcB \mcU \mcG,\\
			&\text{ then solve $ \tilde u \in \mbC^N $ such that}\\ 			
			&\qqq \left( \mcA  + \mcB\mcU\mcB^T \right) \tilde u  = \mcF,\\
			&\text{ then solve $ u \in \mbC^N $ such that}\\
			& \left( \mcA   + c\mcM\right)  u 
			= \mcF - \mcB\mcU\mcB^T \tilde u + c\mcM u_g.
		\end{split}
	}
\end{equation}

\begin{equation}\label{equiv_H1_all_g1}
	\boxed{
		\begin{split}
			& \text{If $ \dim \mbC_0 \not = 0 $ and $ c > 0 $,} \\ 
			& \text{ find the eigenvectors $ \mcH $ with zero eigenvalue}\\
			& \qqq \left( \mcA + \mcB\mcU\mcB^T\right)  u = \lambda \mcM u,\\
			& \text{ find $ u_g \in\mbC^N $ such that }\\
			&\qqq \left( \mcA  + \mcB \mcU \mcB^T + \mcM\mcH\mcH^T\mcM \right) u_g = \mcB \mcU \mcG,\\
			&\text{ then	solve $ \tilde u \in\mbC^N $ such that }\\
			&\qqq \left( \mcA + \mcB\mcU\mcB^T + \mcM\mcH\mcH^T\mcM \right)  \tilde u = \mcF,\\
			&\text{ then solve $ u \in \mbC^N $ such that}\\
			& \left( \mcA + \mcB\mcU\mcB^T + c\mcM\right)  u 
			= \mcF - \mcB\mcU\mcB^T \tilde u + \mcB\mcU\mcG + c\mcM u_g.\\
		\end{split}
	}
\end{equation}

\begin{equation}\label{equiv_H1_2_g1}
	\boxed{
		\begin{split}
			& \text{If $ \dim \mbC_0 \not = 0 $ and $ c > 0 $,} \\ 
			& \text{ find the eigenvectors $ \mcH $ with zero eigenvalue}\\
			& \qqq \left( \mcA + \mcB\mcU\mcB^T\right)  u = \lambda \mcM u,\\
			& \text{ find $ u_g \in\mbC^N $ such that }\\
			&\qqq \left( \mcA  + \mcB \mcU \mcB^T + \mcM\mcH\mcH^T\mcM \right) u_g = \mcB \mcU \mcG,\\
			&\text{ then solve $ \tilde u \in\mbC^N $ such that }\\
			&\qqq \left( \mcA + \mcB\mcU\mcB^T + \mcM\mcH\mcH^T\mcM \right)  \tilde u = \mcF,\\
			&\text{ then solve $ u \in \mbC^N $ such that}\\
			& \left( \mcA   + c\mcM\right)  u 
			= \mcF - \mcB\mcU\mcB^T \tilde u + c\mcM u_g.
		\end{split}
	}
\end{equation}

When $ c=0 $ in this case $ \dim \mbC_0 = 0 $,
we find that $ u_g $ is not necessary in the last equation of
the equivalent problem \eqref{equiv_H0_all_g1}.
Then the last two equations are enough in the case
$ \dim \mbC_0 = 0 $ and $ c=0 $:
\begin{equation}\label{equiv_H0_c0_1_g1}
	\boxed{
		\begin{split}
			&\text{If  $ \dim \mbC_0 = 0 $ and $ c=0 $,} \\
			&\text{	solve $ \tilde u \in\mbC^N $ such that }\\
			&\qqq \left( \mcA  + \mcB\mcU\mcB^T \right) \tilde u  = \mcF,\\
			&\text{ then solve $ u \in \mbC^N $ such that}\\
			& \left(\mcA + \mcB\mcU\mcB^T \right) u 
			= \mcF  - \mcB\mcU\mcB^T \tilde u
			+ \mcB \mcU \mcG.\\
		\end{split}
	}
\end{equation}

For the case $ \dim \mbC_0 = 0 $ and $ c = 0 $,
the equivalent problem \eqref{equiv_H0_c0_g0} for $ \mcG = 0 $
can be also modified for  the general case $ 0 \not = \mcG \in \IM \mcB^T $.
The solution $ u\in \mbC^N $ in the system \eqref{mixed_matrix_g1} is
\begin{equation}\label{u_H0_c0}
	\begin{split}
		u =  u^{(1)} + u_g.
	\end{split}
\end{equation}
Here $ u^{(1)} \in \mbC_1 $ is unique the solution of 
$ \mcA u^{(1)} = \mcM f^{(1)}  $
and $ u_g \in \mbC_2 $ is unique the solution of 
$ \mcB\mcU\mcB^T u_g = \mcB \mcU \mcG  $.
Let $ u^{(2)} \in \mbC_2 $ denote the unique solution of 
$ \mcB\mcU\mcB^T u^{(2)} = \mcM f^{(2)}  $.
Then for a positive number $ \alpha_1 >0 $,
we can verify that
\begin{equation}\label{u_1_s}
	\begin{split}
		u_1 \triangleq  u^{(1)} + u_g +\frac{1}{\alpha_1} u^{(2)}\\
	\end{split}
\end{equation}
is the solution of the problem
\begin{equation}\label{u_1}
	\begin{split}
		(\mcA + \alpha_1 \mcB\mcU\mcB^T ) u_1 &=  \mcF + \alpha_1 \mcB\mcU\mcG, \\
		(\mcA + \alpha_1 \mcB\mcU\mcB^T ) (u^{(1)} + u_g + \frac{1}{\alpha_1} u^{(2)}) 
		&=  \mcM f^{(1)} + \mcM f^{(2)} + \alpha_1 \mcB\mcU\mcG,\\
		\underbrace{\mcA  u^{(1)} - \mcM f^{(1)} }_{\mcM\mbC_1} 
		&= \underbrace{
			\left( \mcM f^{(2)} - \mcB\mcU\mcB^T  u^{(2)} \right)
			+ \alpha_1\left(  \mcB\mcU\mcG - \mcB\mcU\mcB^T  u_g \right)  
		}_{\mcM\mbC_2}.
	\end{split}
\end{equation}
Similarly, we can also verify that
\begin{equation}\label{u_2}
	\begin{split}
		u_2 \triangleq  u^{(1)} + u_g +\frac{1}{\alpha_2} u^{(2)}\\
	\end{split}
\end{equation}
is the solution of the equation
\begin{equation}\nonumber %label{}
	\begin{split}
		(\mcA + \alpha_2 \mcB\mcU\mcB^T ) u_1 &=  \mcF + \alpha_2 \mcB\mcU\mcG.
	\end{split}
\end{equation}
Finally, we can obtain the solution \eqref{u_H0_c0} 
by the linear combination of 
the solutions of two equations \eqref{u_1_s} and \eqref{u_2}:
\begin{equation}\nonumber%\label{}
	\begin{split}
		u \equiv u^{(1)} + u_g = 
		\frac{\alpha_1 u_1 - \alpha_2 u_2 }{\alpha_1 - \alpha_2}.
	\end{split}
\end{equation}
Then we can obtain the following equivalent problem 
for the solution $ u \in \mbC^N $ in the system \eqref{mixed_matrix_g1}
in the case
$ \dim \mbC_0 = 0 $ and $ c = 0 $:
\begin{equation}\label{equiv_H0_c0_2_g1}
	\boxed{
		\begin{split}
			&\text{If $ \dim \mbC_0 = 0 $ and $ c = 0 $,
				for $ \alpha_1 \not = \alpha_2 >0 $,}\\
			&\text { solve $ u_1, u_2 \in\mbC^N $ such that}\\
			&\;\; \left( \mcA  + \alpha_1 \mcB\mcU\mcB^T \right) u_1 
			= \mcF + \alpha_1 \mcB\mcU\mcG \;\; \text{  and  }\\
			&\;\; \left( \mcA  + \alpha_2 \mcB \mcU \mcB^T \right) u_2 
			= \mcF + \alpha_2 \mcB\mcU\mcG,\\
			& \text{ then the solution is}\\
			&\qqq \qqq u = \frac{\alpha_1 u_1 - \alpha_2 u_2 }{\alpha_1 - \alpha_2}.
		\end{split}
	}
\end{equation}
In the equivalent problem \eqref{equiv_H0_c0_1_g1} and \eqref{equiv_H0_c0_2_g1},
there is no need to compute an explicit $ u_g $.

In the end of this section,
we summarize the equivalent problems in Table \ref{table_g1} for 
the solution $ u \in \mbC^N $ 
in the system \eqref{mixed_matrix_g1}
for the case $ \mcG \not = 0  $.
\begin{table}[ht]  %\label{ }	
	\centering  
	\begin{tabular}{| c| c | c |}
		\hline
		& $ c>0 $ & $ c=0 $ \\
		\hline
		$ \dim \mbC_0  = 0 $ 
		& \eqref{equiv_H0_all_g1} \eqref{equiv_H0_c1_g1}
		& \eqref{equiv_H0_all_g1}  
		\eqref{equiv_H0_c0_1_g1} \eqref{equiv_H0_c0_2_g1}\\
		\hline
		$ \dim \mbC_0 \not = 0 $ 
		& \eqref{equiv_H1_all_g1} \eqref{equiv_H1_2_g1} 
		& \tabincell{c}{no solution or \\ no unique solution}   \\	
		\hline	
	\end{tabular}
	\caption{The equivalent problems 
		for the system \eqref{mixed_matrix} in the case $ \mcG \not = 0  $.}
	\label{table_g1}
\end{table}

\section{The discretization using finite element complexes}\label{sec_5}

In this section, 
we use the finite element spaces on discrete de Rham complex
to discretize the constrained problem
\begin{equation}\label{mixed_dd}
	\begin{split}
		(d^k)^* d^k \vu + c \vu + d^{k-1} \vp &= \vf, \\
		(d^{k-1})^* \vu &= \vg.
	\end{split}
\end{equation}
We will show the relation between its discrete weak formulation 
and the matrix system \eqref{mixed_matrix}.
We use the framework 
of finite element exterior calculus (FEEC) in 
\cite{ arnold2018finite, arnold2006finite,arnold2010finite}
to study the discrete problem of \eqref{mixed_dd}.
We refer the readers to these references 
for more details.

\subsection{The weak formulation}
We study the weak formulation of \eqref{mixed_dd} 
in the framework of the Hilbert complex.
The de Rham complex is a typical example of the Hilbert complex
when the operators are differential operators
and the spaces are the corresponding function spaces.
Let us consider a Hilbert complex $ (W,d) $.
The $ d^k $ is a closed densely defined operator
from $ W^k $ to $ W^{k+1} $ 
and its domain is denoted by $ V^k $.
Then the corresponding  domain complex is
\begin{equation}\label{Hcomplex}
	\begin{split}
		\xymatrix{
			0 \ar[r] & V^0 \ar[r]^{d^0} 
			& V^1 \ar[r]^{d^1} 
			& \cdots \ar[r]^{d^{n-1}} 
			& V^n \ar[r]& 0.
		}
	\end{split}
\end{equation}
The adjoint operator of $ d^k $ is denoted by
$ (d^k)^*: W^{k+1} \to W^{k} $
and is defined as
\begin{equation}\nonumber %label{}
	\begin{split}
		\left\langle  d^k\vu \, ,\, \vv \right\rangle 
		=\left\langle \vu \, ,\,  (d^k)^* \vv \right\rangle,
	\end{split}
\end{equation}
if $ \vu \in V^{k} $ or $ \vv \in W^{k+1} $
vanishes near the boundary.
Its domain is a dense subset of $ W^{k+1} $
and denoted by $ V^*_{k+1} $.
Then we have the dual complex
\begin{equation}\nonumber %\label{Hcomplex_dual}
	\begin{split}
		\xymatrix{
			0 \ar[r] & V^*_n \ar[r]^{(d^{n-1})^*} 
			& V^*_{n-1} \ar[r]^{(d^{n-2})^*} 
			& \cdots \ar[r]^{(d^{0})^*} 
			& V^*_0 \ar[r]& 0.
		}
	\end{split}
\end{equation}
The range and the null spaces of the differential operators
are denotes by
\begin{equation}\nonumber %label{}
	\begin{split}
		\mfB^k = d^{k-1}V^{k-1},
		\qq \mfZ^k = \mcN(d^k),
		\qq \mfB^*_k = (d^{k})^* V^*_{k+1},
		\qq \mfZ^*_k = \mcN((d^{k-1})^*).
	\end{split}
\end{equation}
The cohomology space is denoted by $ \mcH^k = \mfZ^k/\mfB^k $
and the space of harmonic $ k $-forms is denoted by
$ \mfH^k = \mfZ^k \cap \mfZ^*_k $.

In this paper,
we focus on the Hilbert complex with \emph{compactness property},
i.e. the inclusion $ V^k\cap V^*_k \subset W^k $ is compact for each $ k $.
In this case, 
the Hilbert complex is closed and Fredholm \cite[Theorem 4.4]{arnold2018finite}.
Then we have 
$$ \mcH^k \cong \mfH^k, $$
and their dimensions are finite.
There is the following Hodge decomposition \cite[Theorem 4.5]{arnold2018finite}
\begin{equation}\nonumber %label{}
	\begin{split}
		V^k = \mfB^k \oplus \mfH^k \oplus \mfZ^{k\bot_V}.
	\end{split}
\end{equation}
Here $ \mfZ^{k\bot_V} = \mfB^*_k \cap V^k $.

The problem \eqref{mixed_dd} involves a segment on the complex \eqref{Hcomplex}:
\begin{equation}\label{Hcomplex_segment}
	\begin{split}
		\xymatrix{
			V^{k-1} \ar[r]^{d^{k-1}} 
			& V^k \ar[r]^{d^k} 
			& V^{k+1}.
		}
	\end{split}
\end{equation}
The weak formulation of \eqref{mixed_dd} is:
find $ \vu \in  V^k $ such that
\begin{equation}\label{mixed_d_weak}
	\begin{split}
		\left\langle d^k \vu \, ,\,  d^k \vv \right\rangle
		+ c \left\langle  \vu \, ,\,  \vv \right\rangle
		+ \left\langle d^{k-1}\vp \, ,\,  \vv \right\rangle 
		&= \left\langle \vf \, ,\,  \vv \right\rangle
		\qqq \vv \in V^k,\\
		\left\langle \vu \, ,\,  d^{k-1}\vq \right\rangle 
		&= \left\langle \vg \, ,\,  \vq \right\rangle
		\qqq  \vq \in V^{k-1}.
	\end{split}
\end{equation}

\subsection{The approximation for the mixed formulation} \label{approx_Hodge}

Let $ V^{k-1}_h$ and $ V^{k}_h $ denote the 
finite element spaces of $ V^{k-1}$ and $ V^{k} $ 
on the complex segment \eqref{Hcomplex_segment}, respectively.
The corresponding discrete weak formulation of \eqref{mixed_d_weak} is:
find $ \vu_h \in  V^k_h $ such that
\begin{equation}\label{mixed_d_dis}
	\begin{split}
		\left\langle d^k \vu_h \, ,\,  d^k \vv_h \right\rangle
		+ c \left\langle  \vu_h \, ,\,  \vv_h \right\rangle
		+ \left\langle d^{k-1}\vp_h \, ,\,  \vv_h \right\rangle 
		&= \left\langle \vf \, ,\,  \vv_h \right\rangle
		\qqq \; \vv_h \in V^k_h,\\
		\left\langle \vu_h \, ,\,  d^{k-1}\vq_h \right\rangle 
		&= \left\langle \vg \, ,\, \vq_h \right\rangle
		\qqq\; \vq_h \in V^{k-1}_h.\\
	\end{split}
\end{equation}
The finite element spaces are required to have the following properties:
\begin{itemize}
	\item $ 1^\circ $ Approximation property:
	\begin{equation}\nonumber %label{}
		\begin{split}
			\textfor \;\vu \in V^j, \qq
			\lim_{h\to 0}\inf_{\vv_h \in V^j_h}\nm{\vu-\vv_h} = 0,
			\qq  j = k-1 \;\textand\; k.
		\end{split}
	\end{equation}
	\item $ 2^\circ $ Subcomplex property:
	$ d^{k-1} V^{k-1}_h \subset V^k_h $ and $ d^k V^{k}_h \subset V^{k+1}_h $,
	i.e. the three spaces form a complex segment:
	\begin{equation}\label{complex_segment_discrete}
		\begin{split}
			\xymatrix{
				V^{k-1}_h \ar[r]^{d^{k-1}} 
				& V^{k}_h \ar[r]^{d^{k}} 
				& V^{k+1}_h.	
			}
		\end{split}
	\end{equation}
	\item $ 3^\circ $ Bounded cochain projections
	$ \pi^j:V^j \to V^j_h $, $ j = k-1,k,k+1 $:
	the following diagram commutes:
	\begin{equation}\nonumber %label{}
		\begin{split}
			\xymatrix{
				V^{k-1} \ar[r]^{d^{k-1}} \ar[d]^{\pi_h^{k-1}}
				& V^{k} \ar[r]^{d^{k}} \ar[d]^{\pi_h^{k}}
				& V^{k+1} \ar[d]^{\pi_h^{k+1}}\\
				V^{k-1}_h \ar[r]^{d^{k-1}} 
				& V^{k}_h \ar[r]^{d^{k}} 
				& V^{k+1}_h.	
			}
		\end{split}
	\end{equation}
	And $ \pi^j_h $ is bounded, 
	i.e. there exists a constant $ c $ such that
	$ \nm{\pi^j_h \vv} \leq c\nm{\vv} $
	for all $ \vv \in V^j $.
	
\end{itemize}

The discrete differential operator $ d^k_h $
is defined as the restriction of $ d^k $ 
on the finite dimensional space $ V^j_h $:
\begin{equation}\nonumber %\label{dh_restriction}
	\begin{split}
		d^k_h = d^k|_{V^k_h}: V^k_h \to V^{k+1}_h.
	\end{split}
\end{equation} 
As the dimension of $ V^j_h $ is finite,
the discrete operator $ d^j_h $ is bounded.
Then its adjoint $ (d^j_h)^* $ is everywhere defined
and the spaces $ V^*_{jh} $ coincide with 
$ W^j_h = V^j_h $.
Then,
for $ \vu_h \in V^{k+1}_h $,
$ (d^k_h)^*\vu_h \in V^k_h $
can be presented as:
\begin{equation}\nonumber %label{}
	\begin{split}
		\left\langle  (d^k_h)^*\vu_h \, ,\, \vv_h \right\rangle 
		=\left\langle \vu_h \, ,\,  d^k_h  \vv_h \right\rangle,
	\end{split}
\end{equation}
for all  $ \vv \in V^k_h $.
The range and the null space of the discrete differential operators
are denotes by
\begin{equation}\nonumber %label{}
	\begin{split}
		\mfB^k_h = d^{k-1}_h V^{k-1}_h,
		\qq \mfZ^k_h = \mcN(d^k_h),
		\qq \mfB^{*}_{kh} = (d^{k}_h)^*V^{k+1}_h
		\qq \text{and} \qq \mfZ^{*}_{kh} = \mcN((d^{k-1}_h)^*).
	\end{split}
\end{equation}
The space of discrete harmonic $k$-forms is denoted by
$ \mfH^k_h = \mfZ^k_h \cap \mfZ^{*}_{kh} $.
Then there is the discrete Hodge decomposition \cite[(5.6)]{arnold2018finite}:
\begin{equation}\label{discrete_Hodge_decomposition}
	\begin{split}
		V^k_h =  \mfH^k_h \oplus \mfB^{*}_{kh} \oplus \mfB^k_h .
	\end{split}
\end{equation}
By  \cite[Theorem 5.1]{arnold2018finite},
if the finite element spaces $ V^{k-1}_h $ and $ V^{k}_h $ satisfy these properties,
there is an isomorphism between $ \mfH^k_h $ and $ \mfH^k $.
This means that 
if the dimension of $ \mfH^k $ is a limited number,
the dimension of $ \mfH^k_h $ keeps stable with the mesh varying.

\subsection{The matrix forms of the discrete problem }

In the subsection,
we construct the coefficient matrices involved in the discrete form \eqref{mixed_d_dis}.
Let $ M = \dim V^{k-1}_h $, $ N = \dim V^{k}_h $,
$ \left\lbrace \vq_{h,i}\right\rbrace_{i=1}^M  $
and
$ \left\lbrace \vv_{h,i}\right\rbrace_{i=1}^N  $
be the base of the space $ V^{k-1}_h $ and $ V^k_h $,
respectively.
We use light letters to denote
the coefficient vectors of the 
discrete functions  in $ V^{k-1} $ and $ V^{k} $.
According to \eqref{mixed_d_dis},
the coefficient matrices
$ \mcA \in \mbC^{N\times N} $ and $ \mcB \in \mbC^{N \times M} $,
the mass matrices
$ \mcM \in \mbC^{N\times N} $ 
and the right-hand sides $ \mcF \in \mbC^N $ and $ \mcG \in \mbC^M $
are constructed as follows:
\begin{equation}\label{def_matrix}
	\begin{split}
		\mcA_{ij} &= \left\langle d^k \vv_{h,i} \, ,\, d^k \vv_{h,j} \right\rangle,
		\qqq \qq
		\mcB_{ij} = \left\langle  d^{k-1}\vq_{h,j} \, ,\,  \vv_{h,i} \right\rangle,\\
		\mcM_{ij} &= \left\langle  \vv_{h,i} \, ,\,  \vv_{h,j} \right\rangle, 
		\qqq\;
		\mcF_i = \left\langle  \vf  \, ,\,  \vv_{h,i} \right\rangle
		\qq \textand \qq
		\mcG_i = \left\langle  \vg \, ,\,  \vq_{h,i} \right\rangle.
	\end{split}
\end{equation}
Then the matrix system  of 
the discrete weak formulation \eqref{mixed_d_dis} is
\begin{equation}\nonumber %label{}
	\begin{split}
		\left( \mcA + c \mcM \right)  u + \mcB p &= \mcF,\\
		\mcB^T  u &= \mcG.
	\end{split}
\end{equation}
The following theorem shows that
the theoretical results in Section \ref{sec_2}
also match the system we define in this section.
Then the discrete problems \eqref{mixed_d_dis} 
can be solved using the equivalent problems 
that we constructed in Section \ref{sec_3} and \ref{sec_4}.
\begin{theorem}
	The matrices $ \mcA, \mcM \in \mbC^{N\times N} $ and $ \mcB \in \mbC ^{N \times M} $
	that we define in this section satisfy Assumption \ref{AMB0}:
	\begin{equation}\nonumber %label{}
		\begin{split}
			\mcA \mcM^{-1}\mcB = 0.
		\end{split}
	\end{equation}
\end{theorem}

\begin{proof}
	For a vector $ p \in \mbC^{M } $,
	let $ \vp_h \in V^{k-1}_h $ be the function
	that is represented by $ p $ and the base
	$ \left\lbrace \vq_{h,i}\right\rbrace_{i=1}^M $:
	\begin{equation}\nonumber %label{}
		\begin{split}
			\vp_h = \sum_{i=1}^M p_i \vq_{h,i}.
		\end{split}
	\end{equation}
	As there is the property $ d^{k-1}V^{k-1}_h \subset V^k_h $ on the complex \eqref{complex_segment_discrete},
	we know that $ \vu_h \triangleq d^{k-1} \vp_h \in V^k_h $.
	Let $ u\in \mbC^N $ be the coefficient vector of $ \vu_h $ in the base 
	$ \left\lbrace \vv_{h,i}\right\rbrace_{i=1}^N $:
	\begin{equation}\nonumber %label{}
		\begin{split}
			\vu_h = \sum_{i=1}^N u_i \vv_{h,i}.
		\end{split}
	\end{equation}
	Then we have
	\begin{equation}\label{B_present}
		\begin{split}
			\left\langle  d^{k-1}\vp_h \, ,\,  \vv_{h,i} \right\rangle
			= \left\langle  \vu_h \, ,\,  \vv_{h,i} \right\rangle
			\qqq \vv_{h,i} \in V^k_h.
		\end{split}
	\end{equation}
	Associated with the definition of the matrix $ \mcB $ in \eqref{def_matrix},
	the matrix representation of \eqref{B_present} is 
	\begin{equation}\label{B_present_2}
		\begin{split}
			\mcB p = \mcM u.
		\end{split}
	\end{equation}
	For this $ \vu_h $, we have
	\begin{equation}\label{A_present}
		\begin{split}
			\left\langle  d^{k}\vu_h \, ,\,  d^{k}\vv_{h,i} \right\rangle
			= \left\langle d^{k} d^{k-1} \vp_h \, ,\,  d^{k}\vv_{h,i} \right\rangle = 0
			\qqq \vv_{h,i} \in V^k_h,
		\end{split}
	\end{equation}
	where we use the property $ d^{k} d^{k-1} = 0 $ on complex.
	Associated with the definition of the matrix $ \mcA $ in \eqref{def_matrix},
	the matrix representation of \eqref{A_present} is 
	\begin{equation}\label{A_present_2}
		\begin{split}
			\mcA u = 0.
		\end{split}
	\end{equation}
	Combining \eqref{B_present_2} and \eqref{A_present_2}, 
	we have $ \mcA \mcM^{-1}\mcB p = 0 $ for any $ p \in \mbC^M $. 
	Then we obtain the conclusion.
\end{proof}

\subsection{The discrete Hodge Laplacian problem}

If we replace $ \vg $ by $ \vp $ in the constrained problem \eqref{mixed_dd},
we obtain the Hodge Laplacian problem:
\begin{equation}\nonumber %label{}
	\begin{split}
		(d^k)^* d^k \vu + c \vu + d^{k-1} \vp &= \vf, \\
		\vp - (d^{k-1})^* \vu &= 0.
	\end{split}
\end{equation}
Its weak formulation is:
find $ \vu \in  V^k $ such that
\begin{equation}\nonumber %\label{Hodge_eig}
	\begin{split}
		\left\langle d^k \vu \, ,\,  d^k \vv \right\rangle
		+ c \left\langle  \vu \, ,\,  \vv \right\rangle
		+ \left\langle d^{k-1}\vp \, ,\,  \vv \right\rangle 
		&= \left\langle \vf \, ,\,  \vv \right\rangle
		\qqq \vv \in V^k,\\
		\left\langle \vp \, ,\,  \vq \right\rangle 
		- \left\langle \vu \, ,\,  d^{k-1}\vq \right\rangle 
		&= 0
		\qqq \qqq \;\,  \vq \in V^{k-1},\\
	\end{split}
\end{equation}
The discrete weak formulation by the finite element spaces 
$ V^{k-1}_h $ and $ V^k_h $ is:
find $ \vu_h \in  V^k_h $ such that
\begin{equation}\label{Hodge_eig}
	\begin{split}
		\left\langle d^k \vu_h \, ,\,  d^k \vv_h \right\rangle
		+ c \left\langle  \vu_h \, ,\,  \vv_h \right\rangle
		+ \left\langle d^{k-1}\vp_h \, ,\,  \vv_h \right\rangle 
		&= \left\langle \vf \, ,\,  \vv_h \right\rangle
		\qq \;\, \vv_h \in V^k_h,\\
		\left\langle \vp_h \, ,\, \vq_h \right\rangle
		-\left\langle \vu_h \, ,\,  d^{k-1}\vq_h \right\rangle 
		&= 0
		\qqq \qqq \; \vq_h \in V^{k-1}_h.\\
	\end{split}
\end{equation}
In this discrete problem,
it involves a mass matrix $ \mcM_{k-1} \in \mbC^{M \times M} $ for the space $ V^{k-1}_h $:
\begin{equation}\nonumber %label{}
	\begin{split}
		(\mcM_{V^{k-1}})_{ij} = \left\langle  \vq_{h,i} \, ,\,  \vq_{h,j} \right\rangle.
	\end{split}
\end{equation}
The matrix system of the discrete Hodge problem \eqref{Hodge_eig}  is 
\begin{equation}\nonumber %label{}
	\begin{split}
		\left( \mcA + c \mcM \right)  u + \mcB p &= \mcF,\\
		\mcM_{V^{k-1}} p  - \mcB^T  u &= 0.
	\end{split}
\end{equation}
Substituting the second equation into the first one,
we have 
\begin{equation}\label{Hodge_matrix}
	\begin{split}
		\left( \mcA+ \mcB \mcM_{V^{k-1}}^{-1}\mcB^T + c \mcM  \right)  u = \mcF.
	\end{split}
\end{equation}
This is the equation in the equivalent problems
with the choice
\begin{equation}\nonumber %label{}
	\begin{split}
		\mcU = \mcM_{V^{k-1}}^{-1}.
	\end{split}
\end{equation}

%\clearpage

\section{Some aspects before numerical experiments} \label{sec_6}

In this sections, 
we consider some problems that are involved in numerical experiments.

\subsection{How to solve the equivalent problems}

The system of the equation \eqref{Hodge_matrix}
is a full matrix
as the term $ \mcM_{V^{k-1}}^{-1} $.
Because $ \mcM_{V^{k-1}} $ is a mass matrix,
it and its inverse usually have quite good conditions.
If we choose a $ \mcU $ that has similar spectrum to $ \mcM_{V^{k-1}} $,
the distribution of the spectrum of the new equation 
can be roughly similar to the equation \eqref{Hodge_matrix}.
If the such $ \mcU $ is sparse enough,
the system of the new equation become sparse.
Then many iterative methods and preconditioning techniques
can be easily applied to them.
When constructing the equivalent problems for the matrix systems
in the previous sections,
we have only one assumption on $ \mcU $, 
i.e.  $ \mcU $ is SPD.
In our paper \cite{Aux_iter},
the choice $ \mcU = \alpha \mcI $,
where $ \alpha $ is proper positive number 
and $ \mcI \in \mbC^{M \times M} $ is the identity matrix,
is efficient in solving many problems.
In \cite{Aux_iter},
We also use multigrid method and ILU-type preconditioners 
to accelerate their convergences of the following 
linear equation and eigenvalue problem:
\begin{align} %label{}
	\left( \mcA + \mcB\mcU\mcB^T + c\mcM\right) u &= \mcF, \label{linear_AB}\\
	\left( \mcA + \mcB\mcU\mcB^T\right)  u &= \lambda \mcM u.\label{eig_AB}
\end{align}

For the eigenvalue problem,
we recommend using the 
Locally Optimal Block Preconditioned Conjugate Gradient Method (LOBPCG) \cite{knyazev2001toward}.
In the equivalent problems \eqref{equiv_H1_all_g1} and \eqref{equiv_H1_2_g1},
we need compute a complete base for the subspace $ \mbC_0 $,
which are the eigenvectors of the zero eigenvalue of \eqref{eig_AB}.
If $ \dim \mbC_0 >1 $,
zero is a multiple eigenvalue of \eqref{eig_AB}.
The LOBPCG can guarantee the entirety of the eigenspace of a multiple eigenvalue.

It also involves a type of equations in the case $ \dim \mbC_0 \not = 0 $:
\begin{equation}\nonumber %label{}
	\begin{split}
		\left( \mcA + \mcB\mcU\mcB^T + \mcM\mcH\mcH^T\mcM \right) u = \mcF.
	\end{split}
\end{equation}
By the following theorem, 
the preconditioner for \eqref{linear_AB} can be also used to this problem.
\begin{theorem}\label{ILU_H1}
	\begin{equation}\nonumber %label{}
		\begin{split}
			\sigma \left( \left( \mcA + \mcB\mcU\mcB^T + \mcM \right)^{-1}
			\left( \mcA + \mcB\mcU\mcB^T + \mcM\mcH\mcH^T\mcM \right) \right) 
			\subset \left[ 1 - \frac{1}{1+ \lambda_{min}} \;,\; 1 \right].
		\end{split}
	\end{equation}
	Here $ \sigma(\cdot) $ denotes the spectrum of a matrix
	and
	$ \lambda_{min} $ is the minimum nonzero eigenvalue of 
	$ \left( \mcA + \mcB\mcU\mcB^T  \right) u = \lambda \mcM u $.
\end{theorem}

\begin{proof}
	
	The eigenvectors in the three eigenpair sets \eqref{A_eigenpairs},
	\eqref{B_pairs} and \eqref{H_pairs}
	form a base for the space $ \mbC^N $.
	It is enough use the vectors in the three sets to prove the theorem.
	
	Letting $ \left(1, \mfu^{(0)}_i \right)  $ be a pair in the set \eqref{H_pairs},
	we have
	\begin{equation}\nonumber %label{}
		\begin{split}
			\left( \mcA + \mcB\mcU\mcB^T + \mcM \right)\mfu^{(0)}_i &= \mcM \mfu^{(0)}_i,\\
			\left( \mcA + \mcB\mcU\mcB^T + \mcM\mcH\mcH^T\mcM \right) \mfu^{(0)}_i &= \mcM \mfu^{(0)}_i.
		\end{split}
	\end{equation}
	Then we obtain
	\begin{equation}\nonumber %label{}
		\begin{split}
			&\left( \mcA + \mcB\mcU\mcB^T + \mcM \right)^{-1}
			\left( \mcA + \mcB\mcU\mcB^T + \mcM\mcH\mcH^T\mcM \right) \mfu^{(0)}_i\\
			&= \left( \mcA + \mcB\mcU\mcB^T + \mcM \right)^{-1} \mcM \mfu^{(0)}_i\\
			& = \mfu^{(0)}_i.
		\end{split}
	\end{equation}
	Letting $ \left(\lambda^{(1)}_i,\mfu^{(1)}_i \right) $ 
	be an eigenpair in the set \eqref{A_eigenpairs}, we have
	\begin{equation}\nonumber %label{}
		\begin{split}
			\left( \mcA + \mcB\mcU\mcB^T + \mcM \right)\mfu^{(1)}_i 
			&= \left(\lambda^{(1)}_i +1 \right) \mcM \mfu^{(1)}_i,\\
			\left( \mcA + \mcB\mcU\mcB^T + \mcM\mcH\mcH^T\mcM \right) \mfu^{(1)}_i
			&= \lambda^{(1)}_i\mcM \mfu^{(1)}_i.
		\end{split}
	\end{equation}
	Then we have
	\begin{equation}\nonumber %label{}
		\begin{split}
			&\left( \mcA + \mcB\mcU\mcB^T + \mcM \right)^{-1}
			\left( \mcA + \mcB\mcU\mcB^T + \mcM\mcH\mcH^T\mcM \right) \mfu^{(1)}_i\\
			=\;&  \lambda^{(1)}_i\left( \mcA + \mcB\mcU\mcB^T + \mcM \right)^{-1} \mcM \mfu^{(1)}_i\\
			=\;& \frac{\lambda^{(1)}_i}{1+ \lambda^{(1)}_i} \mfu^{(1)}_i.
		\end{split}
	\end{equation}
	Similarly, letting $ \left(\lambda^{(2)}_i, \mfu^{(2)}_i \right) $ 
	be an eigenpair in the set \eqref{B_pairs}, we have
	\begin{equation}\nonumber %label{}
		\begin{split}
			\left( \mcA + \mcB\mcU\mcB^T + \mcM \right)^{-1}
			\left( \mcA + \mcB\mcU\mcB^T + \mcM\mcH\mcH^T\mcM \right) \mfu^{(2)}_i
			= \frac{\lambda^{(2)}_i}{1+ \lambda^{(2)}_i} \mfu^{(2)}_i.
		\end{split}
	\end{equation}
\end{proof}

\subsection{The measure of the errors}\label{sec_6_2}

To measure the convergence when using iterative methods,
we use the relative error
\begin{equation}\label{mixd_error}
	\begin{split}
		\frac{\nm{\mcF - \mcB p -\left( \mcA + c\mcM\right) u} + \nm{ \mcG -\mcB^T u}}
		{\nm{\mcF} + \nm{\mcG}}
	\end{split}
\end{equation}
in the last equation of these equivalent problems.
This error involves 
the solution $ p \in \mbC^M $ in the system 
\eqref{mixed_matrix_g0} or \eqref{mixed_matrix_g1}. 
In the case that $ \Ker \mcB $ in Theorem \ref{CM_decomp} does not vanish
or the columns of $ \mcB $ are not full-rank,
the solution $ p \in \mbC^M $ is not unique.
We do not compute $ p $ in these equivalent problems.
However, the term $ \mcB p $ is unique
and it is the component of $ \mcF - (\mcA + c\mcM)u $ 
in the subspace $ \mcM \mbC_2 $.
As the settings in the equivalent problems 
\eqref{equiv_H0_all_g1}, \eqref{equiv_H0_c1_g1},
\eqref{equiv_H1_all_g1} and \eqref{equiv_H1_2_g1},
$ u_g \in \mbC_2 $ is the component of $ u $ in $ \mbC_2 $,
then we have
\begin{equation}\nonumber %label{}
	\begin{split}
		\mcB p = \mcM f^{(2)} - c\mcM u_g = \mcB\mcU\mcB^T \tilde u - c\mcM u_g,
	\end{split}
\end{equation}
where we also use the decomposition for $ \mcF $ \eqref{F_tilde_u} or \eqref{F_tilde_u_H1}.
Then we have $ \mcB p = \mcB\mcU\mcB^T \tilde u $ 
in the equivalent problems
\eqref{equiv_H0_all_g0}, \eqref{equiv_H0_c1_g0}, \eqref{equiv_H1_all_g0} 
and \eqref{equiv_H1_2_g0}
as $ u_g = 0 $ in the case $ \mcG = 0 $.

In the case $ \dim \mbC_0 = 0 $ and $ c=0 $,
we do not compute $ \tilde u $ and $ u_g $
in the equivalent problems 
\eqref{equiv_H0_c0_g0} and \eqref{equiv_H0_c0_2_g1}.
In this case, 
$ \mcB p = \mcM f^{(2)} = \mcB\mcU\mcB^T u^{(2)}$,
where $ u^{(2)} $ is the component in 
$ u_1 $ and $ u_2 $ of 
\eqref{u_1_H0} and \eqref{u_2_H0} or
\eqref{u_1_s} and \eqref{u_2}.
We do not compute this $ u^{(2)} $ explicitly,
but we can obtain it by the combination 
of $ u_1 $ and $ u_2 $ 
in \eqref{equiv_H0_c0_g0} and \eqref{equiv_H0_c0_2_g1}.:
\begin{equation}\nonumber %label{}
	\begin{split}
		u^{(2)} = \frac{\alpha_1 \alpha_2}{\alpha_2 - \alpha_1}\left(u_1 - u_2 \right).
	\end{split}
\end{equation}

\subsection{The penalty method}

Let $ \mcS \in \mbC^{M \times M} $ be a matrix that is easy to invert.
For the equation
\begin{equation}\label{penalty_method_S}
	\begin{split}
		\left( \mcA + c \mcM + 
		{\varepsilon} \mcB \mcS^{-1} \mcB^T \right)  u_{\varepsilon} 
		= \mcF + {\varepsilon} \mcB \mcS^{-1} \mcG,
	\end{split}
\end{equation}
its solution $ u_\varepsilon $ tend to the solution $ u $ 
on the system \eqref{mixed_matrix} 
as the penalty parameter $ \varepsilon $ goes to infinity,
i.e.
\begin{equation}\nonumber %label{}
	\begin{split}
		u_{\varepsilon} \to u  \text{  as  } \varepsilon \to \infty.
	\end{split}
\end{equation}
This method is widely used to solving the constrained problems.

The penalty parameter can not be arbitrarily large
in actual computations because of the limit of computer precision.
Consequently, 
the error of the penalty solution $ u_\varepsilon $
can not be arbitrarily small.
Furthermore, the large parameters results that 
the condition of the system \eqref{penalty_method_S} becomes bad. 
In the numerical experiments, 
we will compare the solutions of the penalty method and the exact solutions.

\subsection{The inconsistent initial data $ \mcG $}

In actual applications,
if the columns of $ \mcB $ are not full-rank,
it is possible that
the data $ \mcG $ does not satisfy 
the basic requirement $ \mcG \in \IM \mcB^T $.
This results in that
there does not exist a $ u\in \mbC^N $ such that $ \mcB^T u = \mcG $.
By Theorem \ref{CM_decomp},
there is the orthogonal decomposition for $ \mbC^M $:
\begin{equation}\nonumber %label{}
	\begin{split}
		\mbC^M =  \IM \mcB^T \oplus \Ker \mcB .
	\end{split}
\end{equation}
Then the $ \mcG $ can be divided into two orthogonal parts:
\begin{equation}\label{G_decomp}
	\begin{split}
		\mcG = \mcG_c + \mcG_i,
	\end{split}
\end{equation}
where $ \mcG_c \in \IM \mcB^T $ and $ \mcG_i \in \Ker \mcB $.
We call $ \mcG_c $ the consistent part of $ \mcG $
and $ \mcG_i $ the inconsistent part of $ \mcG $.
The inconsistent data can be introduced 
through some inevitable ways,
for instance, numerical quadrature, data collection, machine precision and so on.
Of course, 
there is no solution mathematically
when the inconsistent data $ \mcG_i \not = 0  $.
However, the solution for the consistent part $ \mcG_c $
may be still meaningful,
especially when the consistent part $ \mcG_c $ is the dominant part of $ \mcG $.

As the discussions in Section \ref{sec_3} and \ref{sec_4},
the linear equations in the equivalent problems that we construct
are all well-posed.
What will happen if we still use these equivalent problems
to compute the solution $ u\in \mbC^N $ 
in the system \eqref{mixed_matrix_g1}
in this case $ \mcG \not \in \IM \mcB^T $?

We choose
\begin{equation}\label{U_aI}
	\begin{split}
		\mcU = \alpha \mcI,
	\end{split}
\end{equation}
where $ \alpha$ is a positive number 
and $ \mcI \in \mbR^{M \times M} $ is the identity matrix.
The only place that involves $ \mcG $ in these equivalent problems
is the right hand side $ \mcB \mcU \mcG $.
If $ \mcU $ is chosen as \eqref{U_aI},
we have
\begin{equation}\label{BG_c}
	\begin{split}
		\mcB \mcU \mcG = \alpha \mcB \left( \mcG_c + \mcG_i \right) 
		= \alpha \mcB \mcG_c  \equiv \mcB \mcU \mcG_c.
	\end{split}
\end{equation}
By this result,
we know that only the consistent part $ \mcG_c $ in $ \mcG $
has effect on the solutions of these equivalent problems.
Then the solution $ u\in\mbC^N $ in these equivalent problems
with the choice \eqref{U_aI} for $ \mcU $ 
is the solution of the system
\begin{equation}\label{mixed_matrix_Gc}
	\begin{split}
		(\mcA + c\mcM)u + \mcB p &= \mcF \\
		\mcB^T u &= \mcG_c.
	\end{split}
\end{equation}

In this situation, 
we can not use the error \eqref{mixd_error} to measure the convergence 
of these equivalent problems
as the term $ \nm{ \mcG -\mcB^T u} $ in \eqref{mixd_error}
never converge to zero 
because of the inconsistent component $ \mcG_i \not = 0 $.
It seems that we should replace this term by $ \nm{ \mcG_c -\mcB^T u} $.
However, we do not compute the explicit $ \mcG_c $.
By Theorem \ref{BTGU},
for $ \mcG_c \in \IM \mcB^T $,
$ \mcB^T u = \mcG_c $ is equivalent to
$ \alpha \mcB \mcB^T u = \alpha\mcB \mcG_c \equiv \alpha\mcB \mcG $
with  choice \eqref{U_aI} for $ \mcU $.
Then in this case,
the error $ \nm{ \mcG_c -\mcB^T u} $ 
can be replaced by
\begin{equation}\nonumber %label{}
	\begin{split}
		\alpha \nm{ \mcB \mcB^T u - \mcB  \mcG_c}
		\equiv \alpha \nm{ \mcB \mcB^T u - \mcB \mcG}.
	\end{split}
\end{equation}
Then we obtain an error to measure the convergence in this situation:
\begin{equation}\nonumber %label{}
	\begin{split}
		\frac{\nm{\mcF - \mcB p -\left( \mcA + c\mcM\right)  u} + \alpha\nm{ \mcB\mcG - \mcB\mcB^T u}}
		{\nm{\mcF} + \nm{\mcG}}.
	\end{split}
\end{equation}

For the penalty method \eqref{penalty_method_S},
we choose $ \mcS = \mcI $ when the $ \mcG $ is inconsistent.
In this case, the penalty method \eqref{penalty_method_S} becomes
\begin{equation}\label{penalty_vare}
	\begin{split}
		\left( \mcA + c \mcM + 
		\varepsilon \mcB \mcB^T \right)  u_{\varepsilon} 
		= \mcF + \varepsilon \mcB \mcG
	\end{split}
\end{equation}
The solution $ u_\varepsilon $ also tends to 
the solution $ u\in\mbC^N $ in the system \eqref{mixed_matrix_Gc}
as the penalty parameter $ \varepsilon $ goes to infinity.

%\clearpage

\section{Numerical experiments}\label{sec_7}

In this section,
we take the constrained problems 
of the $ \nabla \times \nabla \times $ and $ \nabla \left( \nabla \cdot\right) $
operators as examples to verify the equivalent problems 
that we constructed in the previous sections.
The two operators are the  $ k=1 $ and $ k=2 $ forms of the $ d^* d $ operator 
on the $ \mbR^3 $ complex, respectively:
\begin{equation}\nonumber %\label{R3_complex}
	\begin{split}
		\xymatrix{
			0 \ar[r] 
			& H^1(\Omega) \ar[r]^{\nabla}        \ar[d]^{\Pi^Q_h}
			& \vH(\textcurl,\Omega) \ar[r]^{\nabla \times} \ar[d]^{\Pi^\vE_h}
			& \vH(\textdive,\Omega) \ar[r]^{\nabla \cdot}  \ar[d]^{\Pi^\vF_h}
			& L^2(\Omega) \ar[r]                  \ar[d]^{\Pi^S_h}
			& 0\\
			0\ar[r] 
			& Q_h \ar[r]^{\nabla} 
			& \vE_h \ar[r]^{\nabla \times} 
			& \vF_h \ar[r]^{\nabla \cdot}	
			& S_h \ar[r]
			& 0.
		}
	\end{split}
\end{equation}
Here the finite element spaces are the first family 
of \Nedelec element spaces \cite{nedelec1980mixed}.
The node element space $ Q_h $ 
and the edge element space $ \vE_h $ 
are used to discretize the constrained Maxwell problems
\begin{equation}\nonumber %label{}
	\begin{split}
		\nabla \times \nabla \times \vu + c \vu + \nabla p &= \vf, \\
		\nabla \cdot \vu &= \vg,
	\end{split}
\end{equation}
while $ \vE_h $ and the face element space $ \vF_h $ 
are used for the constrained grad-div problems
\begin{equation}\nonumber %label{}
	\begin{split}
		-\nabla ( \nabla \cdot \vu) + c \vu + \nabla \times  \vp &= \vf, \\
		\nabla \times \vu &= \vg.
	\end{split}
\end{equation}
We take three different three-dimensional domains
$ \Omega_1 $, $ \Omega_2 $ and $ \Omega_3 $, shown in Figure \ref{Domains}.
The first one is a cube
$ \left[ 0 \;,\; \pi \right]^3 $.
The second is the same cube but with a tunnel
$ \left[ \frac{\pi}{4} \;,\; \frac{3\pi}{4} \right]^2 \times \left[ 0 \;,\; \pi \right] $.
The third is the cube with a void
$  \left[ \frac{\pi}{4} \;,\; \frac{3\pi}{4} \right]^3 $ inside.
\begin{figure}
	\centering
	\begin{minipage}{4cm}
		\includegraphics[height=4.5cm,width=5cm]{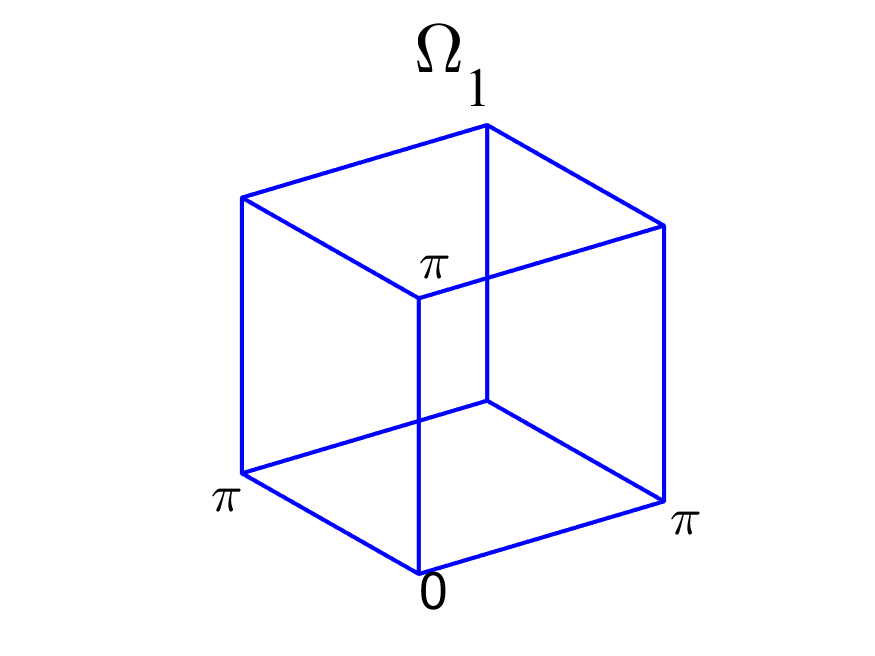}
	\end{minipage}
	\begin{minipage}{4cm}
		\includegraphics[height=4.5cm,width=5cm]{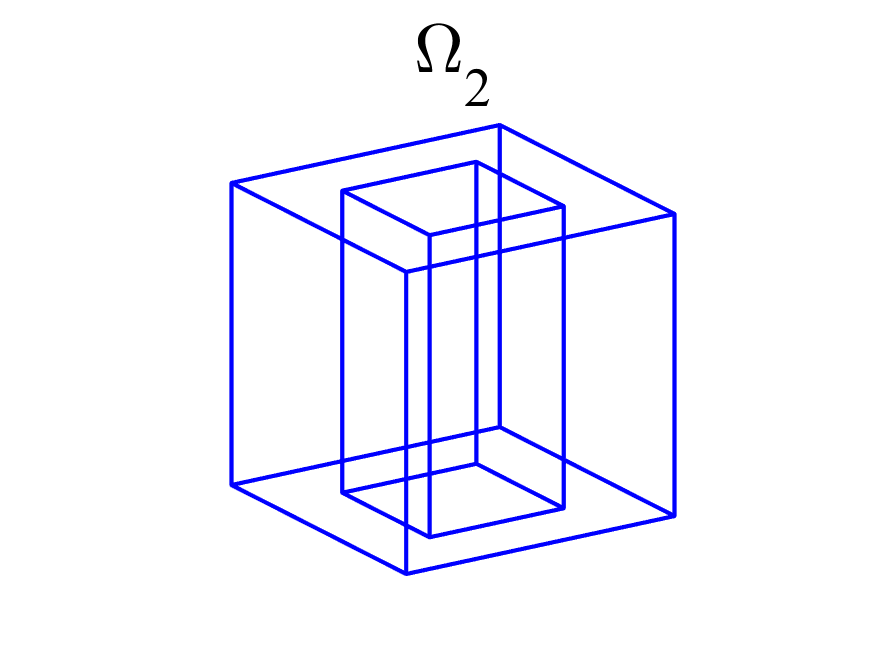}
	\end{minipage}
	\begin{minipage}{4cm}
		\includegraphics[height=4.5cm,width=5cm]{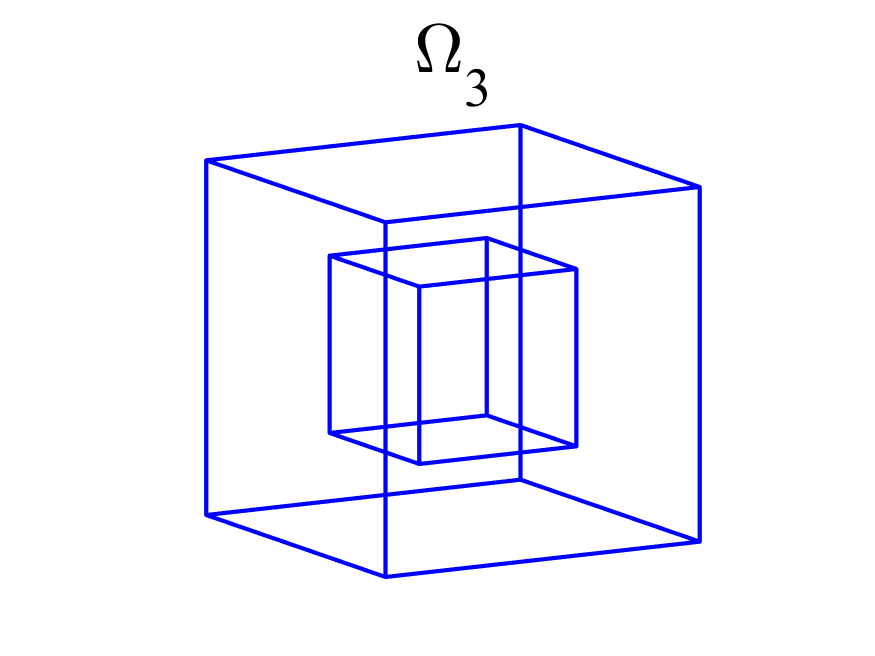}
	\end{minipage}
	\caption{The three 3D domains. $ \Omega_1 $ is a cube $ \left[ 0 \;,\; \pi \right]^3 $.
		$ \Omega_2 $ is the cube with a tunnel
		$ \left[ \frac{\pi}{4} \;,\; \frac{3\pi}{4} \right]^2 \times \left[ 0 \;,\; \pi \right] $.
		$ \Omega_3 $ is the cube with a void
		$  \left[ \frac{\pi}{4} \;,\; \frac{3\pi}{4} \right]^3 $ inside.}
	\label{Domains}
\end{figure}
The last two domains result in nontrivial $ \mbC_0 $ for the two constrained problems,
respectively.

We use uniformly cubic mesh for all the problems.
The mesh size is set to $ h = \frac{\pi}{32} $ uniformly.
We use the first order \Nedelec to construct 
the matrix system of the discrete weak formulation \eqref{mixed_d_dis}.
The matrix $ \mcU $ is chosen as in \cite{Aux_iter}:
\begin{equation}\nonumber %label{}
	\begin{split}
		\mcU = \frac{5}{h^3}\mcI,
	\end{split}
\end{equation}
where $ \mcI \in \mbR^{M\times M} $ is the identity matrix.
We use the Preconditioned Conjugate Gradient method (PCG)
to compute the linear systems
and use the LOBPCG method
to compute the eigenvalue problem 
in the equivalent problems.
We use the ILU(0) preconditioner to accelerate the algorithms.
By Theorem \ref{ILU_H1},
for the equation with the term $ \mcM \mcH\mcH^T \mcM $,
the ILU(0) factors are generated by the sparse matrix 
$ \mcA + \mcB \mcU \mcB^T + \mcM $.
We take random data for the exact solution $ u $ and $ p $
to generate the right hand side $ \mcF $ and $ \mcG $.
We use the error \eqref{mixd_error} to measure the convergence 
of the last equation in the equivalent problems
and use relative errors for the other equations and the eigenvalue problem.
The stopping criterion for the last equation is set to $ 10^{-10} $.
As the last equation depends on the solution of other equations,
the stopping criteria for other equations and eigenvalue problems
should be smaller than the last one.
We set them to $ 10^{-11} $.

\subsection{Example 1}

The first numerical example is the mixed Maxwell problem 
with Dirichlet boundary condition on the domain $ \Omega_1 $:
\begin{equation}\label{example_equ_1}
	\begin{split}
		\nabla \times \nabla \times \vu + \nabla  \vp 
		&= \vf \qq \;\text{in}\; \qq \Omega_1, \\
		\nabla \cdot \vu 
		&= \vg \qq \,\text{in}\; \qq \Omega_1,\\
		\vn \times \vu 
		&= 0 \qq\, \text{on} \qq \partial\Omega_1,\\
		\vp &= 0 \qq \,\text{on} \qq \partial\Omega_1.\\
	\end{split}
\end{equation}
The sizes of the matrices in its discrete system
are $ N = 92256 $ and $ M = 29791 $.
In this problem,
the matrix $ \mcB $ are full-rank.
We can use the direct method to solve the exact solution of its matrix system.
We also use the direct method to solve the solution
of the penalty method \eqref{penalty_vare}.
Figure \ref{penalty_compare} shows the the relative error of 
the penalty solution $ u_\varepsilon $ 
compared with the exact solution.
The error of $ u_\varepsilon $ goes down 
with the penalty parameter $ \varepsilon $ becoming large at first.
After some point,
the error goes up with $ \varepsilon $ continuing becoming large.
The reason of this phenomenon is that
the machine precision is limited 
and the parameter is too large.
\begin{figure}[!htb]
	\begin{center}
		\includegraphics[width=3in]{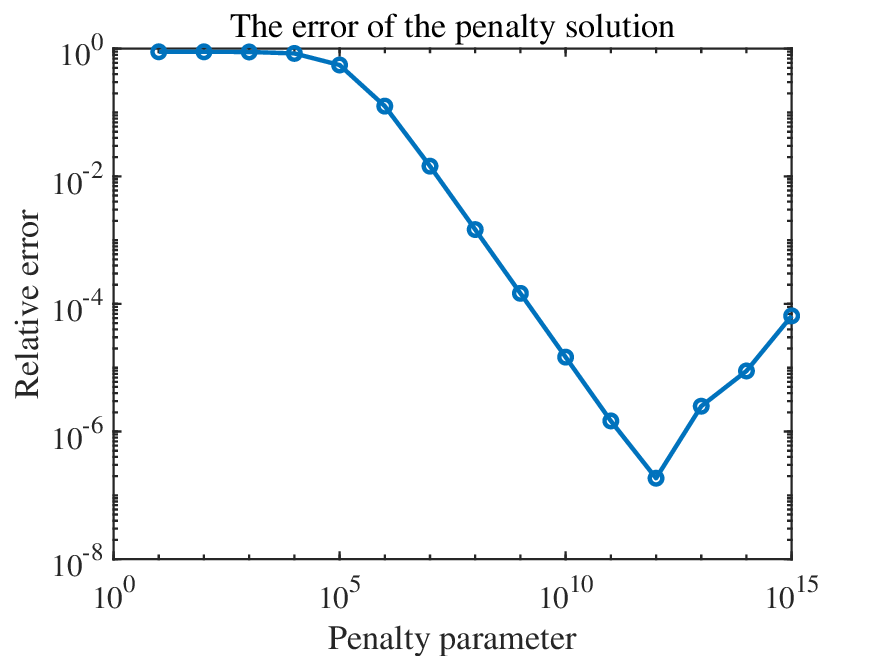}
	\end{center}
	\caption{The relative error of the solution $ u_\varepsilon $
		of penalty method \eqref{penalty_vare} 
		for the discrete problem of \eqref{example_equ_1}
		with the exact solution.}
	\label{penalty_compare}
\end{figure}

As $ \dim \mbC_0 = 0 $ and $ c = 0 $ in this problem,
we use the equivalent problems 
\eqref{equiv_H0_c0_1_g1} and \eqref{equiv_H0_c0_2_g1} 
to solve its matrix system, respectively.
When using the equivalent problem \eqref{equiv_H0_c0_2_g1},
we use the simultaneous iteration for the two equations.
The error is measured after each iteration.
Figure \ref{example_1} shows the convergence histories
in solving the equivalent problems with PCG.
The results illustrate that 
the ILU(0) preconditioner reduces the iteration counts.
\begin{figure}
	\centering
	\begin{minipage}{6cm}
		\includegraphics[height=5cm,width=6.0cm]{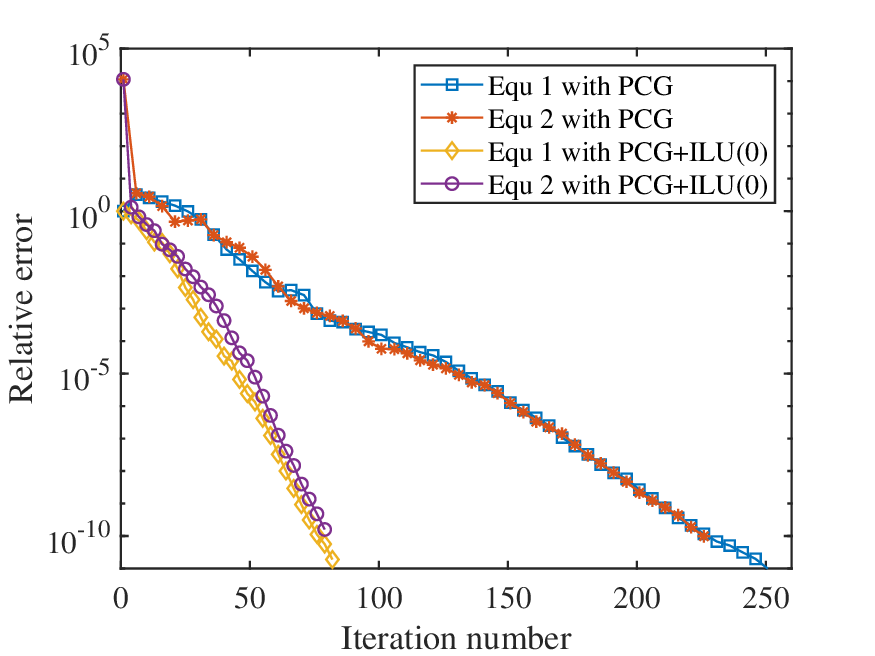}
	\end{minipage}
	\begin{minipage}{6cm}
		\includegraphics[height=5cm,width=6.0cm]{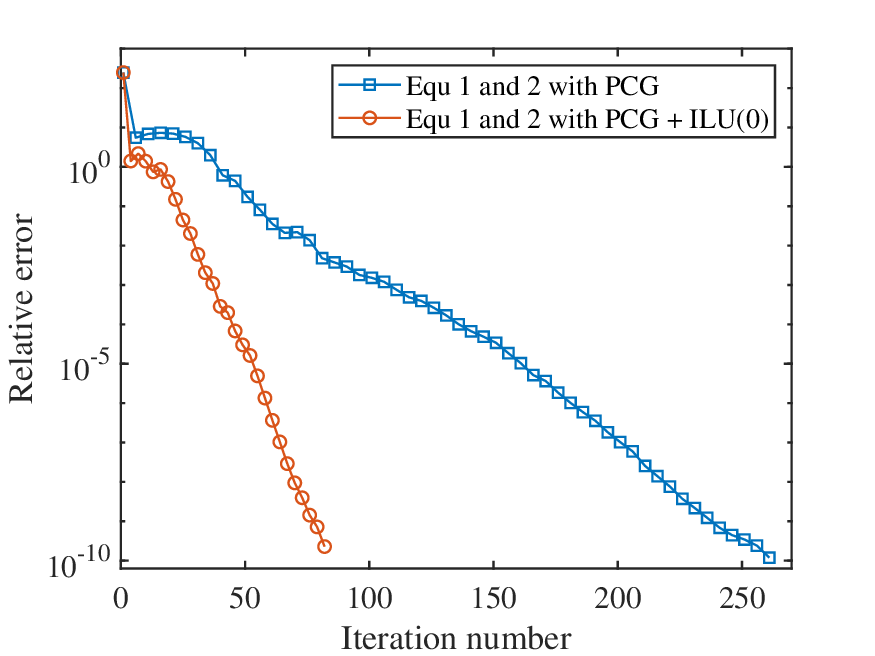}
	\end{minipage}
	\caption{The convergence histories of the equivalent problem 
		\eqref{equiv_H0_c0_1_g1} (left) and \eqref{equiv_H0_c0_2_g1} (right)
		in solving the discrete problems of the equation \eqref{example_equ_1}.}
	\label{example_1}
\end{figure}

% \clearpage

\subsection{Example 2}

The second example is also the mixed Maxwell equation on the domian $ \Omega_1 $
but with Neumann boundary condition:
\begin{equation}\label{example_equ_2}
	\begin{split}
		\nabla \times \nabla \times \vu  + \nabla  \vp 
		&= \vf \qq \;\text{in}\; \qq \Omega_1, \\
		\nabla \cdot \vu 
		&= \vg \qq \, \text{in}\; \qq \Omega_1,\\
		\vn \times (\nabla \times \vu) 
		&= 0 \qq\, \text{on} \qq \partial\Omega_1,\\
		\vn \cdot \nabla\vp &= 0 \,\qq \text{on} \qq \partial\Omega_1.\\
	\end{split}
\end{equation}
The sizes of the matrices in its discrete system
are $ N = 104544 $ and $ M = 35937 $.
In this example,
it is still that
$ \dim \mbC_0 = 0 $ and $ c = 0 $ as the first example.
But the columns of $ \mcB $ are no longer full-rank.
This means that its matrix system is not invertible
and the direct method is invalid for the problem.
We use the equivalent problem 
\eqref{equiv_H0_c0_1_g1} and \eqref{equiv_H0_c0_2_g1} 
to solve this problem, respectively.
Figure \ref{example_2} shows the convergence histories
of PCG method with and without ILU(0) preconditioner.
\begin{figure}
	\centering
	\begin{minipage}{6cm}
		\includegraphics[height=5cm,width=6.0cm]{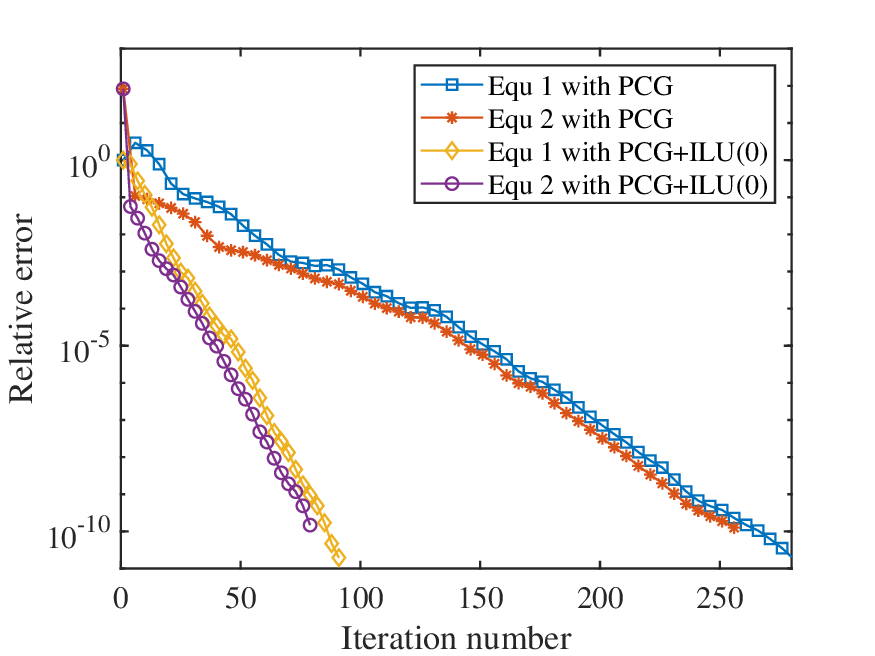}
	\end{minipage}
	\begin{minipage}{6cm}
		\includegraphics[height=5cm,width=6.0cm]{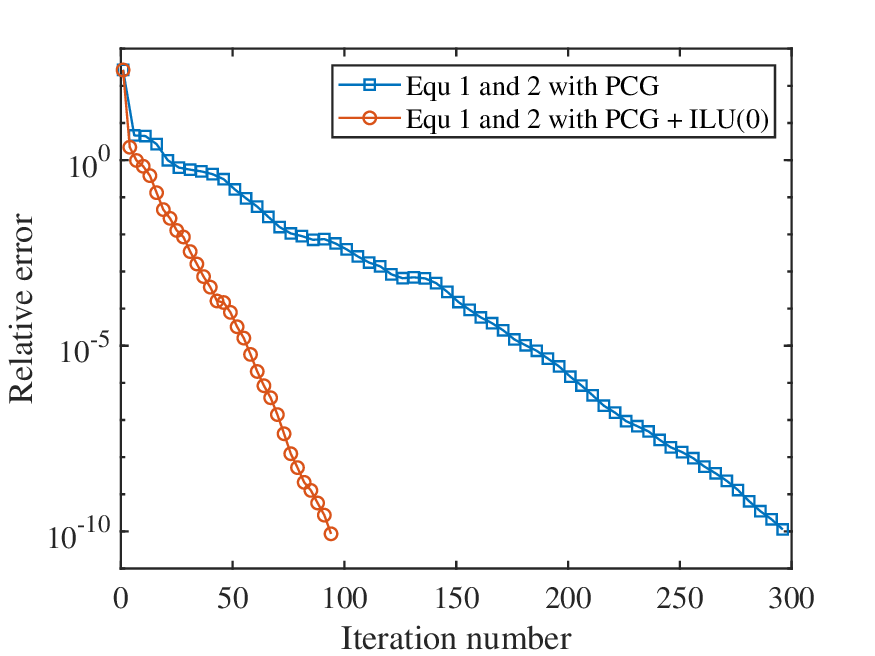}
	\end{minipage}
	\caption{The convergence histories of the equivalent problem 
		\eqref{equiv_H0_c0_1_g1} (left) and \eqref{equiv_H0_c0_2_g1} (right)
		in solving the discrete problems of the equation \eqref{example_equ_2}.}
	\label{example_2}
\end{figure}

\subsection{Example 3}

The third numerical example is the mixed Maxwell equation 
with Neumann boundary condition on the domian $ \Omega_2 $:
\begin{equation}\label{example_equ_3}
	\begin{split}
		\nabla \times \nabla \times \vu +  \vu + \nabla  \vp 
		&= \vf \qq \;\text{in}\, \qq \Omega_2, \\
		\nabla \cdot \vu 
		&= \vg \qq \,\text{in}\; \qq \Omega_2,\\
		\vn \times (\nabla \times \vu) 
		&= 0 \qq\, \text{on} \qq \partial\Omega_2,\\
		\vn \cdot \nabla\vp &= 0 \qq\, \text{on} \qq \partial\Omega_2.\\
	\end{split}
\end{equation}
The sizes of the matrices in its discrete system
are $ N = 81504 $ and $ M = 28512 $.
Because there is a tunnel in this domain,
the cohomology space of this form is not empty
and  $ \dim \mbC_0 = 1 $.
We use the equivalent problem \eqref{equiv_H1_all_g1}
to solve the matrix system of this problem.
Figure \ref{example_3} shows the convergence histories 
of the eigenvalue problem and linear systems in \eqref{equiv_H1_all_g1}.
\begin{figure}
	\centering
	\begin{minipage}{6cm}
		\includegraphics[height=5cm,width=6.0cm]{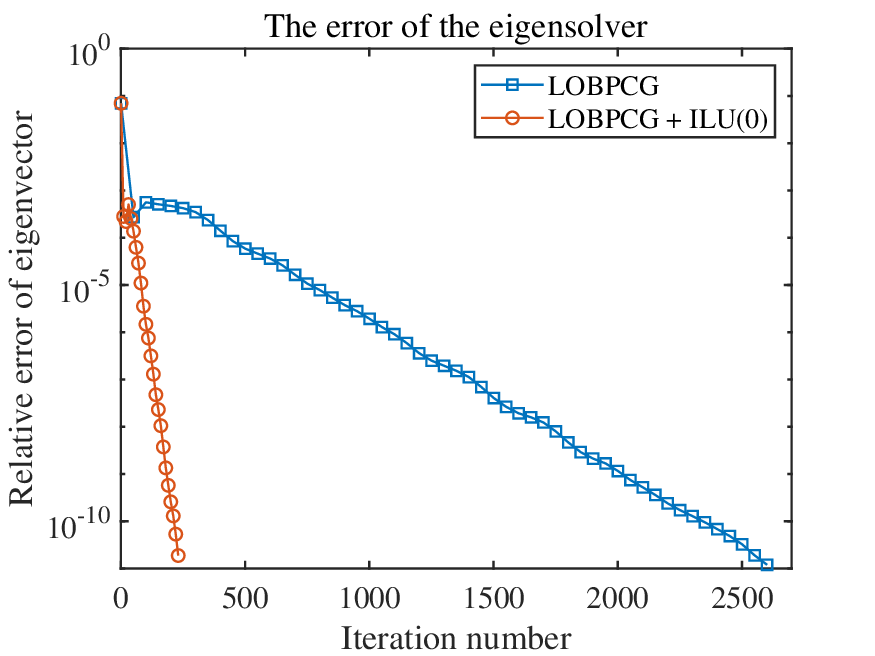}
	\end{minipage}
	\begin{minipage}{6cm}
		\includegraphics[height=5cm,width=6.0cm]{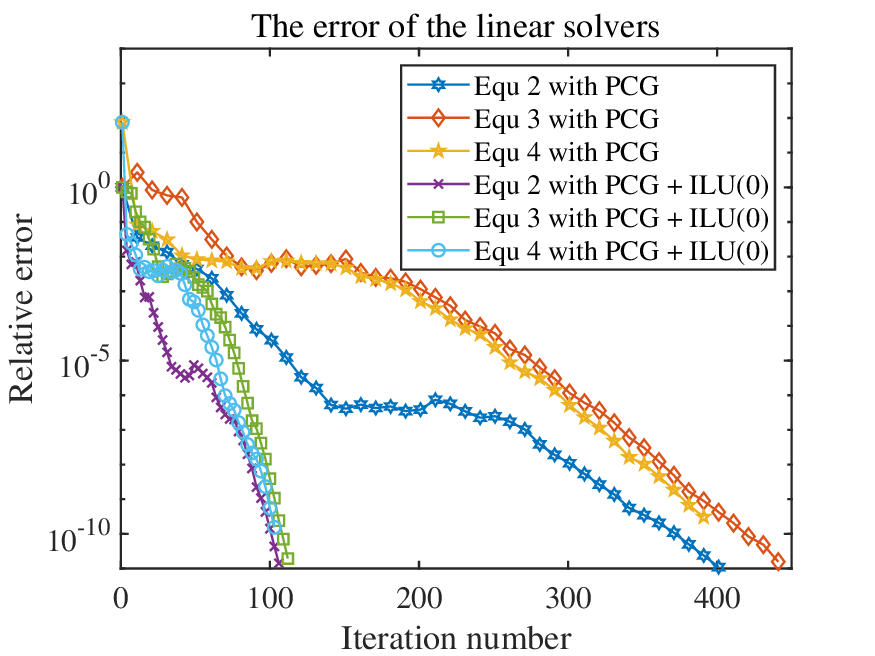}
	\end{minipage}
	\caption{The convergence histories of 
		the eigenvalue problem (left)
		and linear systems (right) of
		the equivalent problem \eqref{equiv_H1_all_g1}
		in solving the discrete problem of the equation \eqref{example_equ_3}.}
	\label{example_3}
\end{figure}

\subsection{Example 4}

The fourth example is the constrained grad-div problem on the domain $ \Omega_1 $:
\begin{equation}\label{example_equ_4}
	\begin{split}
		-\nabla ( \nabla \cdot \vu) + \nabla \times  \vp 
		&= \vf \qq \;\text{in}\; \qq \Omega_1, \\
		\nabla \times \vu 
		&= \vg \qq \,\text{in}\; \qq \Omega_1,\\
		\vn (\nabla \cdot \vu) 
		&= 0 \qq\, \text{on} \qq \partial\Omega_1,\\
		\vn \times (\nabla \times \vp) 
		&= 0 \qq\, \text{on} \qq \partial\Omega_1.\\
	\end{split}
\end{equation}
The sizes of the matrices in its discrete system
are $ N = 101376 $ and $ M = 104544 $.
In this problem, $ \dim \mbC_0 = 0 $ and $ c = 0 $.
We use the equivalent problem 
\eqref{equiv_H0_c0_1_g1} and \eqref{equiv_H0_c0_2_g1} 
to solve this problem, respectively.
Figure \ref{example_4} shows the convergence histories
of PCG method with and without ILU(0) preconditioner.
\begin{figure}
	\centering
	\begin{minipage}{6cm}
		\includegraphics[height=5cm,width=6.0cm]{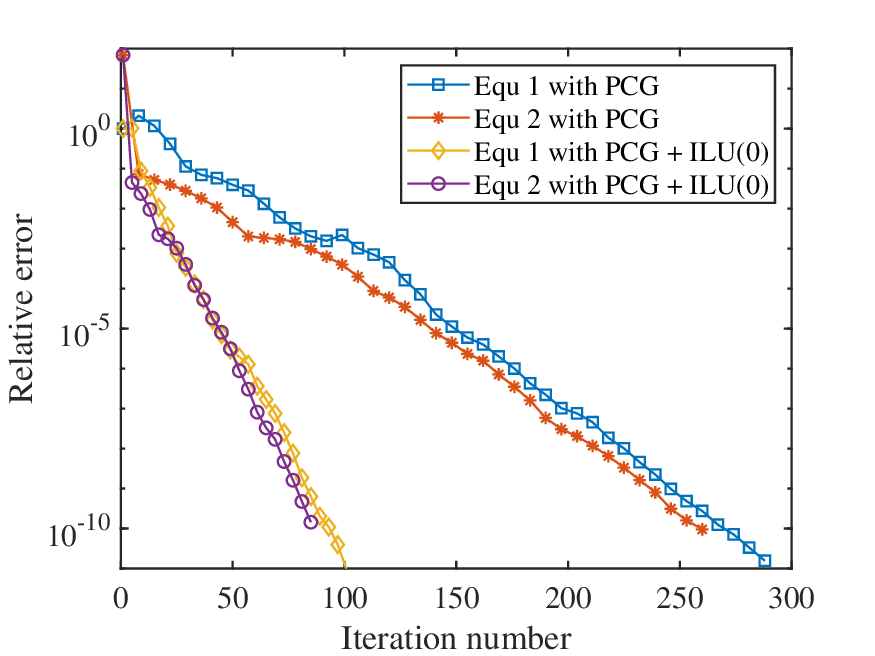}
	\end{minipage}
	\begin{minipage}{6cm}
		\includegraphics[height=5cm,width=6.0cm]{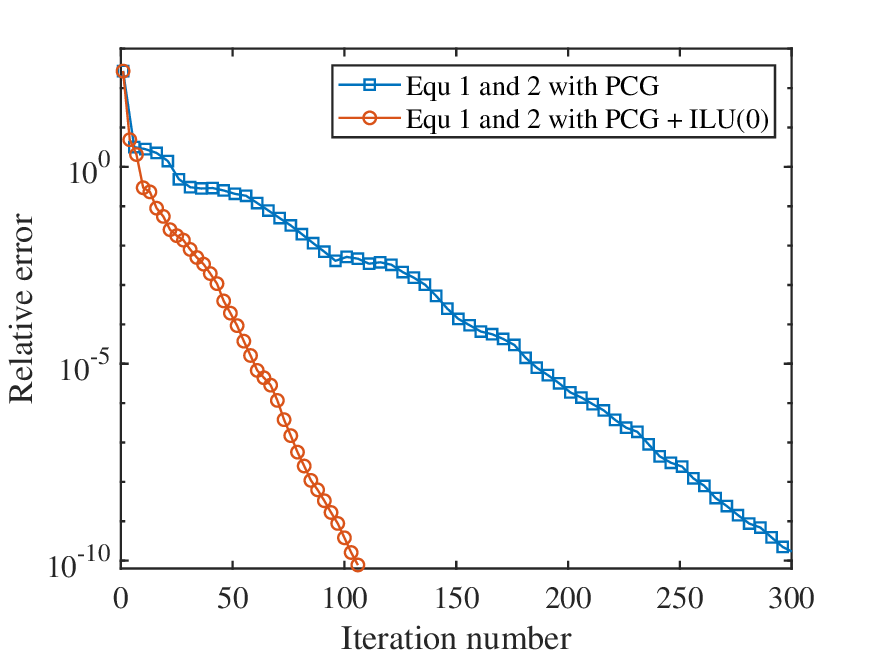}
	\end{minipage}
	\caption{The convergence histories of the equivalent problem 
		\eqref{equiv_H0_c0_1_g1} (left) and \eqref{equiv_H0_c0_2_g1} (right)
		in solving the discrete problems of the equation \eqref{example_equ_4}.}
	\label{example_4}
\end{figure}

\subsection{Example 5}

The fifth numerical example is the constrained grad-div problem 
on the domain $ \Omega_3 $:
\begin{equation}\label{example_equ_5}
	\begin{split}
		-\nabla ( \nabla \cdot \vu) + \vu + \nabla \times  \vp 
		&= \vf \qq \;\text{in}\; \qq \Omega_3, \\
		\nabla \times \vu 
		&= \vg \qq \,\text{in}\; \qq \Omega_3,\\
		\vn (\nabla \cdot \vu) 
		&= 0 \qq \,\text{on} \qq \partial\Omega_3,\\
		\vn \times (\nabla \times \vp) 
		&= 0 \qq \,\text{on} \qq \partial\Omega_3.\\
	\end{split}
\end{equation}
The sizes of the matrices in its discrete system
are $ N = 89856 $ and $ M = 93744 $.
Similar to the third example,
the cohomology space of this form is also not empty
and $ \dim \mbC_0 = 1 $
because of the void in this domain.
We use the equivalent problem \eqref{equiv_H1_all_g1}
to solve the matrix form of this problem.
Figure \ref{example_5} shows the convergence histories
of the eigenvalue problem and linear systems.

\begin{figure}
	\centering
	\begin{minipage}{6cm}
		\includegraphics[height=5cm,width=6.0cm]{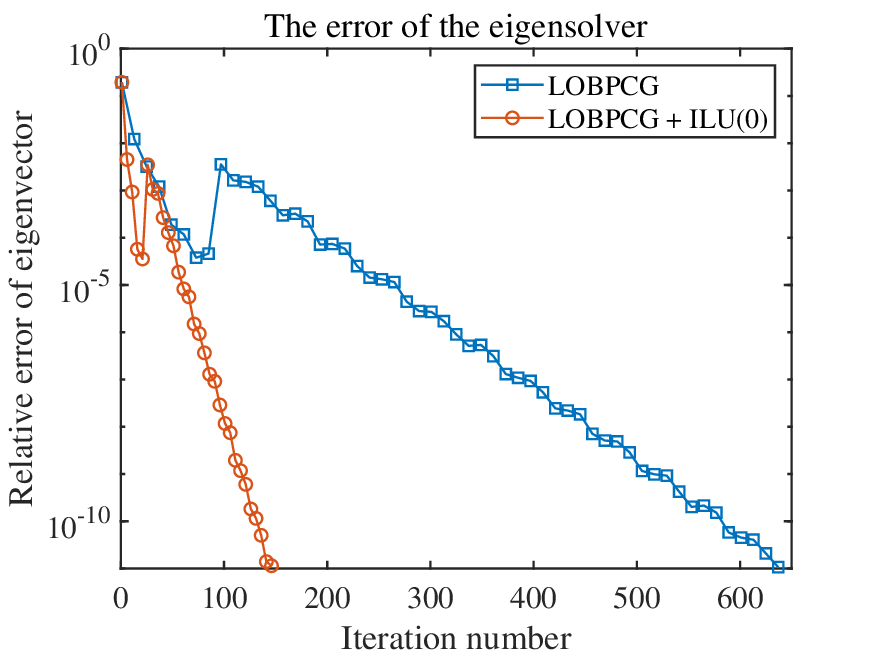}
	\end{minipage}
	\begin{minipage}{6cm}
		\includegraphics[height=5cm,width=6.0cm]{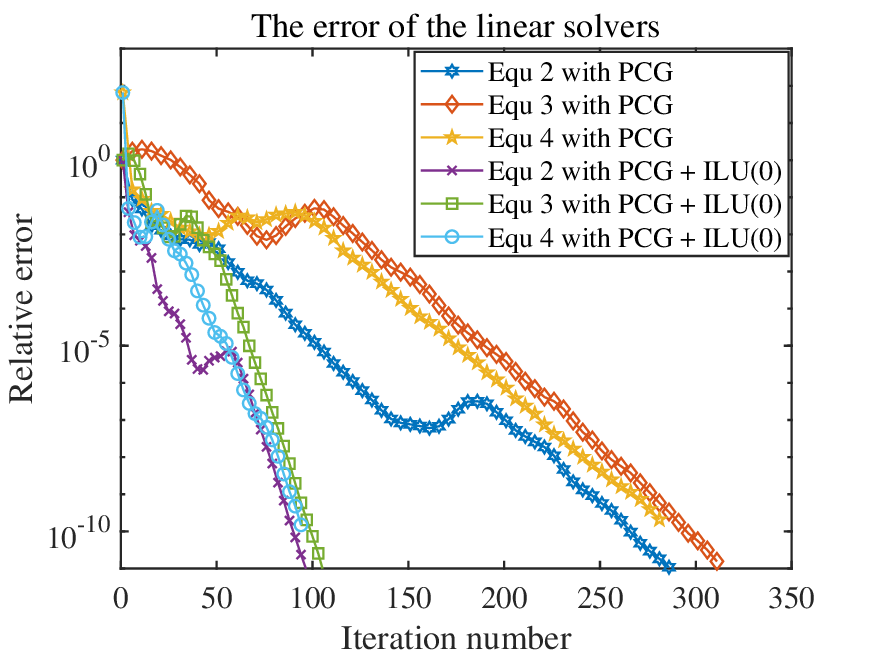}
	\end{minipage}
	\caption{The convergence histories of 
		the eigenvalue problem (left)
		and linear systems (right) of
		the equivalent problem \eqref{equiv_H1_all_g1}
		in solving the discrete problem of the equation \eqref{example_equ_5}.}
	\label{example_5}
\end{figure}

\section{Conclusions}\label{sec_8}

In this paper, 
we consider solving the discrete systems 
of the constrained problems on de Rham complex.
The first difficulty is the poor condition 
of the discrete systems.
Many existing iterative methods and preconditioning techniques
do not work for such systems.
The second difficulty is that 
the systems are non-invertible
in the case that  
the constraint terms do not satisfy 
inf-sup condition,
even if the desired component in the system is still unique.

We discretize the constrained problems
use the finite element complex.
We prove that the matrices in the algebraic system 
satisfy the property $ \mcA\mcM^{-1}\mcB = 0 $.
This property corresponds to the property $ d^2 =0 $ on complex.
This is the extra property of the matrix systems in this paper
compared with general constrained problems.
By this property,
we prove that the explicit Hodge decomposition of the right hand side
can be obtained through solving a Hodge Laplacian problem.
We construct several equivalent problems for 
the discrete systems.
No mater whether the constraint term satisfy
inf-sup condition or not,
only if the component $ \vu_h $ is unique,
we can solve it though some well-posed problems.
Furthermore,
the spectral distributions of the equations contained
in these equivalent problems are Laplace-like,
if the complex is Fredholm.
Then many existing iterative methods
and preconditioning techniques 
can be applied to solving them.
In our paper \cite{Aux_iter},
we have discussed how to solve
the linear equations and eigenvalue problems in these equivalent problems.
This make the large-scalar constrained discrete problem
become easy to solve.

We provide several numerical experiments on $ \mbR^3 $ complex
to verify the equivalent problems.
The numerical results show the capability and efficiency 
of the equivalent problems that we construct.

\

\

\

\textbf{Acknowledgement.} The author would like to thank Prof. Chao Zeng at Nankai University
for many useful discussions.

	%\begin{acknowledgements}
	%\end{acknowledgements}
	
	
\bibliographystyle{plain} 
\bibliography{reference_constrain} 	
	
\end{document}